\documentclass[11pt,a4paper,reqno]{amsart}
\usepackage{amssymb,epic,eepic,hyperref}

\providecommand{\BBb}[1]{{\mathbb{#1}}}

\providecommand{\cal}[1]{{\mathcal{#1}}}   

\newcommand{\Bcirc}{\overset{\lower 1.5pt%
              \hbox{$@,@,@,@,@,\scriptscriptstyle\circ$}}B{}}
\newcommand{\Binfty}{\overset{\lower 1.5pt%
              \hbox{$@,@,@,@,@,\scriptscriptstyle\infty$}}B{}}
\newcommand{\bigdot}{\mathbin{\raise.65\jot\hbox{$\scriptscriptstyle\bullet$}}}

\newcommand{\C}{{\BBb C}}

\newcommand{\Cb}{C_{\operatorname{b}}}

\newcommand{\Dm}{\BBb D}

\newcommand{\dual}[2]{\langle\,#1,\,#2\,\rangle}

\newcommand{\erd}{\overset{\lower 1pt\hbox{\large.}}{e}
                  \overset{\lower 1pt\hbox{\large.}}{r}}

\newcommand{\Fcirc}{\overset{\lower 1.5pt%
               \hbox{$@,@,@,@,@,\scriptscriptstyle\circ$}}F{}}
\newcommand{\fracc}[2]{{
                \textstyle\frac{#1}{\raise 1pt\hbox{$\scriptstyle #2$}}}}
\newcommand{\fracp}{\fracc1p}
\newcommand{\fracnp}{\fracc np}

\newcommand{\fracci}[2]{{\frac{#1}{\raise 1pt\hbox{$\scriptscriptstyle #2$}}}}
\newcommand{\fracpi}{\fracci1p}

\newcommand{\grad}{\operatorname{grad}}

\newcommand{\lap}{\operatorname{\Delta}}

\newcommand{\mlap}{-\!\operatorname{\Delta}}

\newcommand{\norm}[2]{\mathinner{\|}#1\,|#2\|}
\newcommand{\Norm}[2]{\mathinner{\bigl\|\,#1\,\big|#2\bigr\|}}

\newcommand{\op}[1]{\operatorname{#1}}

\newcommand{\N}{\BBb N}

\newcommand{\R}{{\BBb R}}
\newcommand{\Rn}{{\BBb R}^{n}}
\newcommand{\Rp}{\overline{{\BBb R}}_+}

{
\end{minipage}%
\par\medskip\noindent}%
\newcounter{enmcount}\renewcommand{\theenmcount}{{\rm\arabic{enmcount}}}

\newcounter{rmcount}\renewcommand{\thermcount}{{\rm\roman{rmcount}}}

\newcounter{Rmcount}\renewcommand{\theRmcount}{{\rm\Roman{Rmcount}}}

\newcommand{\Z}{\BBb Z}

\numberwithin{equation}{section}

\newtheorem{thm}{Theorem}
\numberwithin{thm}{section}
\newtheorem{prop}[thm]{Proposition}
\newtheorem{lem}[thm]{Lemma}

\theoremstyle{definition}

\newtheorem{exmp}[thm]{Example}
 \numberwithin{exercise}{section}
 
\theoremstyle{remark}
\newtheorem{rem}[thm]{Remark}

\begin{document}
\pagestyle{empty}\renewcommand{\thepage}{}

\pagestyle{plain}
\textwidth=36em\textheight=110ex
~\vspace{22ex}
\begin{center}{\bf\Large SEMI-LINEAR BOUNDARY PROBLEMS OF
\\[\jot] 
{\renewcommand{\!}{\kern-0.2truemm} \mathsurround=0.0pt
 C\!O\!\!MP\!O\!SITI\!O\!N~TYP\!E~IN~$\!\pmb{L}_{\pmb{p}}\!$-\!\!REL\!A\!\!TED~S\!P\!\!\!\!A\!\!\!CES
}}
\end{center}

\medskip
\begin{center}
{\sc Jon Johnsen\/}\footnote{partly supported by the Danish Natural Sciences
Research Council, grant no.~11--1221--1 and no.\ 11--9030} \end{center}
\begin{center}{Institute of Mathematical Sciences, Mathematics Department;\\
Universitets\-par\-ken~5, DK-2100 Copenhagen O; Denmark.\\
E-mail: {\tt jjohnsen@math.ku.dk\/}}
\end{center}

\smallskip
\begin{center}
{\sc Thomas Runst\/}\footnote{partly supported by
Deutsche Forschungsgemeinschaft, grant Tr~374/1-1.
\\[2\baselineskip]
{\tt Appeared in Communications in partial differential equations, {\bf 22\/} (1997), no.7--8, 1283--1324.}} 
\end{center}
\begin{center}{Mathematical Institute, Friedrich--Schiller--Universit{\"a}t
Jena;\\ Ernst--Abbe--Platz 1--4, D-07743 Jena; Germany.\\
E-mail: {\tt runst@minet.uni-jena.de \/}}
\end{center}
\baselineskip=1.5\baselineskip

\bigskip
\section{Introduction} \label{intr-sect}
We address the $L_p$-theory of semi-linear boundary problems of the form:
\begin{equation}
  \begin{aligned}
    Au(x)+g(u(x))&= f(x) &&\quad\text{in }&&\Omega,  \\
    Tu(x)&=0 &&\quad\text{on }&&\Gamma:=\partial\Omega.
  \end{aligned}
  \label{cmp-pb}
\end{equation}
Here $\{A,T\}$ defines a linear elliptic problem (specified below),
$g(t)\in \Cb^\infty(\R)$, 
and we seek solutions $u(x)$ with $s$ derivatives in $L_p(\Omega)$, roughly
speaking. 

The purpose is to study effects caused by the non-linearity
$g(u)$, when one wants a \emph{maximal} range of both $s$ and $p$. As a main
result we describe and determine in Theorem~\ref{dreg-thm} ff.\ below a
certain borderline occurring for $s\in\,]1,\fracnp[\,$. 
To our knowledge neither the borderline nor the range $\,]1,\fracnp[\,$ has
been treated before.

\textheight=123ex
Moreover, for each $n\ge6$ and \emph{fixed} $p$ in $[1,\tfrac{n}{3+\sqrt8}]$
the $H^s_p$-theory is split into two parts by the borderline (loosely speaking
$0<s\lesssim3$ and $s\gtrsim\fracnp$). In
particular this is so for the $H^s$-theory when $n\ge12$.

These phenomena actually occur in any dimension when $p$ is taken arbitrarily
in $\,]0,\infty]$. Thus it is advantageous for the \emph{full}
understanding of \eqref{cmp-pb} to use spaces with $p<1$, and this we
do in the framework 
of the Besov and Triebel--Lizorkin spaces, $B^{s}_{p,q}$ and $F^{s}_{p,q}$.

In this context we treat the existence and regularity of solutions, 
with Land\-es\-man--Lazer conditions for the self-adjoint case.

Our methods combine two general investigations in $B^{s}_{p,q}$ and
$F^{s}_{p,q}$ spaces:
(i) Boutet de Monvel's pseudo-differential calculus of linear boundary
problems, which gives the framework for $\{A,T\}$, with \cite{JJ96ell} by the
first author as source (extending works of Grubb and Franke \cite{G3,F2}); and
(ii) estimates of composition operators $u\mapsto g(u)$ in
works of Sickel and the second author \cite{Run86,RuSi96,Sic89}.

\bigskip

The borderline phenomena occur although we assume that $g(t)$ is
real-valued with bounded derivatives of any order, i.e.\ 
\begin{equation}
  g(t) \in \Cb^\infty(\R).
\end{equation} 
Such non-linearities constitute only a narrow class, but on one hand
new insight can be obtained even for these, and on the other hand our
methods do not allow us to go further since a full set of composition
estimates have not yet been established for wider classes.

As motivated above we treat solutions $u(x)$ in the  Besov 
and Triebel--Lizorkin spaces, $B^{s}_{p,q}$ and $F^s_{p,q}$, with $s\in\R$ and
$p$ and $q$ in $\,]0,\infty]$; throughout with $p<\infty$ for
$F^s_{p,q}$, however. Both $u(x)$ and $f(x)$ are assumed \emph{real-valued}.

Recall that e.g.\  H{\"o}lder--Zygmund spaces $C^s_*=B^{s}_{\infty,\infty}$
($s>0$), 
Sobolev--Slobodetski\u\i\ spaces $W^s_p=B^{s}_{p,p}$ ($s\in\R_+\!\setminus\N$,
$1<p<\infty$), Bessel potential spaces $H^s_p=F^{s}_{p,2}$ ($s\in\R$,
$1<p<\infty$) and local Hardy spaces $h_p=F^0_{p,2}$ ($0<p<\infty$),
cf.~\cite{T2,T3}, so that these are covered by our treatment. 

In \eqref{cmp-pb}, $\Omega\subset\Rn$ is a bounded open set with
$C^\infty$-smooth boundary~$\Gamma$ for $n\ge1$. 
$A=\sum_{|\alpha|\le2} a_\alpha(x)D^\alpha$ is an elliptic operator and
the trace operator $T=S_0\gamma_0+S_1\gamma_1$, where
$\gamma_0u=u|_\Gamma$ is restriction to the boundary while
$\gamma_1u=\gamma_0(\vec n\cdot\grad u)$ for 
a unit outward normal vectorfield, $\vec n$, near~$\Gamma$.
For simplicity $A$ is taken of order $2$ and the boundary condition is
homogeneous, so we only need to treat $A_T$, the $T$-\emph{realisation}
of $A$; for this reason $T$ is assumed to be right invertible 
(e.g.\ $T$ could be normal). Moreover, $A$ and $T$ have coefficients in
$C^\infty(\overline{\Omega})$, and the $S_j$ are differential operators in
$\Gamma$ of order $d-j$ for some $d<2$. The class of $T$ is denoted by $r$; by
definition $r=1$ here if $S_1\equiv0$, and else $r=2$.

Finally, $\{A,T\}$ is assumed elliptic in the Boutet de~Monvel calculus
\cite{BM71}; see \eqref{ellip1}--\eqref{ellip2} below.

\subsection*{Review}
Under the assumptions above we deduce three consequences for the non-linear
problem \eqref{cmp-pb}: 
\begin{itemize}
    \item[(i)] (Theorem~\ref{dreg-thm}.) For $(s,p,q)$ belonging to a 
domain $\Dm(A_T+g(\cdot))$, specified below, the condition $Tv=0$ makes sense
and $v\mapsto g(v)$ has order strictly less than $2$ when 
$v(x)$ in $B^{s}_{p,q}$ or $F^{s}_{p,q}$.

In particular $g(\cdot)$ is better behaved than $A_T$ on
$B^{s}_{p,q}$ and $F^{s}_{p,q}$ whenever $(s,p,q)\in
\Dm(A_T+g(\cdot))$. Because the range $1<s<\fracnp$ is included, the
transformation $(s,p,q)\mapsto (s,\fracnp,\fracc nq)$ will for $n\ge2$ take
$\Dm(A_T+g(\cdot))$ into a \emph{non-convex} subset of $\R^3$.

  \item[(ii)] (Theorem~\ref{ireg-thm}.) Given a solution $u(x)$ in
$B^{s_1}_{p_1,q_1}$ for data $f(x)$ in $B^{s_0-2}_{p_0,q_0}$, where
$(s_j,p_j,q_j)\in\Dm(A_T+g(\cdot))$ for both $j=0$ and $1$, then
$u(x)$ also belongs to $B^{s_0}_{p_0,q_0}$, as in the linear case, and
similarly in the $F$-case. 

Using that $A_T$ has a parametrix in the pseudo-differential calculus, this
follows from a bootstrap argument with \emph{varying} integral exponents; even
for $p_0=p_1$ the $p$'s cannot in general for $n\ge4$ be kept fixed because
$\Dm(A_T+g(\cdot))$ is not convex. 

  \item[(iii)] (Theorem~\ref{solv-thm}.) For $(s,p,q)$ in
$\Dm(A_T+g(\cdot))$ and $f(x)$ in $B^{s-2}_{p,q}$ there exists a solution
$u(x)$ in $B^{s}_{p,q}$, and similarly for the $F^{s}_{p,q}$ scale. This is
proved by means of the Leray--Schauder theorem when $A_T$ is invertible,
as well as when $A_T$ is self-adjoint and $f(x)$ satisfies generalised
Landesman--Lazer conditions, cf.~\cite{RoLa95}.

The proof is standard for $s<2$, for then the embedding, say,
$L_\infty\hookrightarrow B^{s-2}_{p,q}$ shows that
$\norm{g(u)}{B^{s-2}_{p,q}}$ is estimated \emph{independently} of $u$ by
$\norm{g}{L_\infty}$. For larger $s$ such a procedure seems impossible, but
we consider $f(x)$ as an element of some $X\supset L_\infty$ to which the
result for $s<2$ applies; the inverse regularity result in (ii)
yields that the found solution belongs to $B^{s}_{p,q}$ or $F^{s}_{p,q}$ as
required. 
\end{itemize}
Throughout the set $\Dm(A_T+g(\cdot))$ is termed the \emph{parameter domain}
of the operator $A_T+g(\cdot)$, cf.~ Figure~\ref{DM-fig}. 
In addition to (i) above, for $T$ of class $2$ we characterise the
largest possible parameter domain (except for the borderline cases, which are
undiscussed here).

\begin{exmp}[General data]   \label{data-ex}
When $\Omega$ is connected in $\Rn$ for $n\ge2$ and $0\in\Omega$,
we get the following:

(a) For $r=1$, take any $A_T+g(\cdot)$, say
$\mlap_{\gamma_0}u+(1+u^2)^{-1}$. With $x=(x',x_n)$, let $f$ be the 
restriction to $\Omega$ of one the distributions
\begin{equation}
 1(x')\otimes\op{pv}(\tfrac{1}{x_n}), \quad
  1(x')\otimes\delta_0(x_n);
  \label{ex1}
\end{equation}
then $f$ is in $B^{\fracpi-1}_{p,\infty}(\overline{\Omega})$ for
$p\in]1,\infty]$, cf.~Example~\ref{tnsr2-ex}. By Theorem~\ref{solv-thm} there is,
whenever $1<p\le\infty$, 
a solution $v_0(x)$ lying in $B^{\fracpi+1}_{p,\infty}(\overline{\Omega})$.

(b) $r=2$. When $A_T=\mlap_{\gamma_1}$ and $g(t)=\frac{\pi}{2}+\arctan t$,
then, 
\begin{equation}
  f(x)=\chi(x)|x|^\alpha \in B^{\fracci np+\alpha}_{p,\infty}
  \quad\text{for each}\quad p\in\,]0,\infty],
  \label{ex2}
\end{equation}
when $-n<\alpha<0$ and $\chi$ is a cut-off function with $\chi(0)=1$,
cf.~\cite{RuSi96}.  Here each $\alpha\ge-2$
yields $\fracnp+2+\alpha>\fracnp$ and hence $(\fracnp+2+\alpha,p,\infty)\in
\Dm(\mlap_{\gamma_1}+g(\cdot))$ if $p$ satisfies $\fracc{n-1}p+1+\alpha>0$, 
and for $\chi$ such that $\int_\Omega f< \pi$ there is then a solution
$v_1(x)$ in 
$B^{\fracci np+\alpha+2}_{p,\infty}(\overline{\Omega})$ according to
Theorem~\ref{solv-thm}. (Even $-n<\alpha<-2$ may be treated for $p$ in
a smaller interval.)

However, when $-2<\alpha\le-1$ the function $f$ in \eqref{ex2} is not in
$B^{t}_{\infty,\infty}$ for $t>-1$, so the existence of $v_1$ is not
provided by \cite{FrRu88,RoRu96}. 
\end{exmp}

\begin{exmp}[Optimal regularity]  \label{ireg-ex}
By Theorem~\ref{ireg-thm} each $v_0$ in (a) of Example~\ref{data-ex}
also belongs to $B^{\fracci1r+1}_{r,\infty}(\overline{\Omega})$ for every
$r\in\,]1,\infty]$.  

That $v_1$ exists in $H^2$ is
known for $-2<\alpha\le-1$ when $n>-2\alpha$, for $f$ is in $L_2$
in such dimensions. However, that $v_1$ is in
$B^{\fracci np+\alpha+2}_{p,\infty}$ is a stronger fact provided by
Example~\ref{data-ex}. For $n\ge6$ 
this even holds for the classical range $p\in [1,\frac{n}{3+\sqrt8}]$, so in
particular, for $\alpha=-2$ and $n=12$ we conclude that $v_1$ belongs to
$H^{6-\varepsilon}$ for $\varepsilon>0$. 

The typical difficulties caused by the boundary of the parameter domain
$\Dm(\mlap_{\gamma_1}+g(\cdot))$ are illustrated in Figure~\ref{DM-fig} below;
especially 
the dotted line indicates that one cannot just `go upwards' to obtain,
say, $v_1\in H^{6-\varepsilon}$. 
\end{exmp}

\subsection*{Other works} There are numerous articles on semi-linear
problems, so we shall only compare results for the one specified in
\eqref{cmp-pb} ff., and thus leave out the more liberal assumptions found on
e.g.~$g$ in many papers. 

Solutions for
$s=1$ or $2$ and $p=2$ have been treated by e.g.\ Landesman and Lazer
\cite{LaLa70}, Ambrosetti and Mancini \cite{AmMa78}, Br{\'e}zis and Nirenberg
\cite{BrNi78} and Robinson and Landesman \cite{RoLa95}, 
and for $p>1$ by Amann, Ambrosetti and Mancini \cite{AAMa78} and Ne\v cas
\cite{Nec83} 
whereas the $B^{s}_{p,q}$ and $F^{s}_{p,q}$ have been dealt with for
$s>\fracnp$ in works of Franke, Runst and Robinson \cite{FrRu88,RoRu96}.

Spaces with $1<s<\fracnp$ have not been treated systematically for
\eqref{cmp-pb} before, so the non-convexity and the borderline of
$\Dm(A_T+g(\cdot))$ in this region should be novelties, together with
its maximality when $T$ has class~$2$. 

The crucial inverse regularity properties of $A_T+g(\cdot)$ in (ii) above do
not as far as we know have any forerunners, not even under further assumptions
on the $(s,p,q)$'s or on $g(t)$. However, the simpler property that $u(x)$ is
in $C^\infty$ when $f(x)$ is so (hypoellipticity) was obtained in
\cite{AAMa78,AmMa78,BrNi78}.

In contrast to this the solvability of \eqref{cmp-pb} has been treated
extensively with some of the original applications of the Leray--Schauder
theorem containing the case $A_T=\lap_{\gamma_0}$ \cite{LeSc34}. In general,
when $A_T$ is invertible, it was assumed in \cite{FrRu88,RoRu96} that the data
$f$ given in $B^{s-2}_{p,q}$ or $F^{s-2}_{p,q}$ for $s>\fracnp$ should also
belong to $B^t_{\infty,\infty}$ for some $t>-1$ when $T$ has class $r=2$. For
$p<\infty$ and $s<\fracnp+1$ this is a serious restriction, which is removed
in our work.

For $A_T$ self-adjoint, the Landesman--Lazer conditions appeared in
\cite{LaLa70} and was further investigated by Hess, Fu\v cik and the
abovementioned 
\cite{Hes74,AAMa78,AmMa78,BrNi78,FuHe78}.
Extensions to slowly decaying $g$ was given in \cite{FuKr77,Hes77,Nec83}, and
more general versions in \cite{RoLa95}; see \cite{RoLa95} for more
references and a survey on the development of solvability conditions,
and in general also \cite{Run90,RoRu96}. 

Here the generalised Landesman--Lazer conditions of
\cite{RoLa95,RoRu96} are extended to the $B^{s}_{p,q}$ and $F^{s}_{p,q}$ with 
$(s,p,q)$ running in the full $\Dm(A_T+g(\cdot))$, including the range
$1<s<\fracnp$; various other improvements in this extension are collected in
Remarks~\ref{orignorm-rem}--\ref{RoRu-rem} below.

\subsection*{Contents} 2.~Main results and
notation, 3.~Composition estimates, 4.~Proof of the regularity
theorem, 5.~The existence results and 6.~Final remarks.

\section{Main Results and Notation} \label{thms-sect}
In general the notation and the spaces are described in
Sections~\ref{notation-ssect}--\ref{spaces-ssect} below, 
so we proceed to present the results.

\bigskip

For convenience, we shall first of all let $E^s_{p,q}$ stand for a space
which can be either $B^s_{p,q}(\overline{\Omega})$ or
$F^s_{p,q}(\overline{\Omega})$. Hereby we avoid repetition when
properties in the $B^s_{p,q}$ spaces carry over verbatim to the
$F^s_{p,q}$ spaces (but $p<\infty$ must be understood in the
$F^s_{p,q}$ case, of course). 

Secondly, $A_T$ will denote the $T$-\emph{realisation} of $A$. That is, for
\begin{equation}
  s>r+\max(\fracp-1,\fracnp-n),
\label{s-cnd}
\end{equation}
where $r=1$ or $r=2$ denotes the class of $T$, the operator $A_T$ acts
like $A$ in the distribution sense and it is defined for those $u\in
E^s_{p,q}$ that satisfy the boundary condition; hence
\begin{gather}
  A_Tu= Au= \sum_{|\alpha|\le 2} a_\alpha D^\alpha u, 
  \label{AT-def}  \\
  D(A_T)=\bigl\{\,u\in E^s_{p,q}\bigm| Tu=0\,\bigr\}=:E^s_{p,q;T}.
  \label{AT-dom}
\end{gather}
For $(s,p,q)=(2,2,2)$ this is just the usual $H^2$-realisation (in
$L_2$), cf.\ \cite[Def.~1.4.1]{G1}.

Thirdly, the problem is then given by the operator
equation
\begin{equation}
  A_Tu+g(u)=f\quad\text{in}\quad E^{s-2}_{p,q},
  \label{AT-pb}
\end{equation}
with $u(x)$ sought in $E^s_{p,q;T}$ for a 
parameter $(s,p,q)$ satisfying \eqref{s-cnd}.  

In our treatment of \eqref{AT-pb} we build on results for the solution
operator for $A_T$ derived in Section~\ref{realisations-sssect} below from 
\cite{JJ96ell}, 
where the Boutet de~Monvel calculus of pseudo-differential boundary
operators is extended to the $B^s_{p,q}$ and $F^s_{p,q}$ spaces. See
also \cite[Ch.~4]{JJ93} for this.

Another basic ingredient is the results for composition (or
Nemytski\u\i) operators $u(x)\mapsto g(u(x))$, written $g(\cdot)$ for short, 
that have been derived in \cite{Sic89} and 
\cite{Run86}; see also \cite{Run85}. For an overview concerning the
Bessel potential spaces see \cite{Sic92}, and for more results
\cite{RuSi96}. 

Once the function $g(t)$ is given, it is natural to ask for the
parameters $(s,p,q)$ such that $T$ and $g(u)$ both make sense on
$E^s_{p,q}$ and such that $g(\cdot)$ respects the continuity
properties of $A$ on $E^s_{p,q}$; i.e.\ we could introduce 
\begin{multline}
  \Dm=\bigl\{\,(s,p,q)\bigm| \text{$T$ and $g(\cdot)$ are bounded
    from $E^s_{p,q}$},
  \\[-1\jot]
  \exists\varepsilon>0: g(E^s_{p,q})\subset 
  E^{s-2+\varepsilon}_{p,q}\,\bigr\},
  \label{M1}
\end{multline}
which would provide a domain of parameters for the  non-linear
operator $A_T+g(\cdot)$ in the
sense that it goes from $E^s_{p,q;T}\subset E^s_{p,q}$ to
$E^{s-2}_{p,q}$ for each $(s,p,q)\in\Dm$\,---\,through
$\varepsilon$, even with a good control over $g(\cdot)$.

However, our results only allow us to treat a slightly smaller set denoted
$\Dm(A_T+g(\cdot))$ and characterised in the following:

\medskip
\begin{thm} \label{dreg-thm}
  Let $(s,p,q)$ be an admissible parameter for which the following
  conditions are fulfilled:
  \begin{alignat*}{2}
    \text{\upn{(i)}}&\qquad s&&>r+\max(\fracp-1,\fracnp-n),
    \\
    \text{\upn{(ii)}}&\qquad s&&>
      \begin{cases}
        0 &\text{for $1\le p<\infty$},
        \\
        \fracnp+\max(-n,-\tfrac{p}{1-p}) &\text{for $p<1$};
      \end{cases}
      \\
      \text{\upn{(iii)}}&\qquad s&&>
      \tfrac{1}{2}(\fracnp+3+\sqrt{(\fracnp-3)^2-8}\,) \quad\text{or}
      \\
      &\qquad  s&&<
      \tfrac{1}{2}(\fracnp+3-\sqrt{(\fracnp-3)^2-8}\,),
      \quad\text{if $\fracnp\ge3+\sqrt8$}.
  \end{alignat*}
  Then \upn{(i)} and \upn{(ii)}--\upn{(iii)}, respectively, assure that
  \begin{align}
    T&\colon B^{s}_{p,q}(\overline{\Omega})\to
    B^{s-d-\fracpi}_{p,q}(\Gamma),\quad
    T\colon F^{s}_{p,q}(\overline{\Omega})\to
    B^{s-d-\fracpi}_{p,p}(\Gamma),
    \label{M2}  \\
    g(\cdot) &\colon E^s_{p,q}\to E^{\sigma}_{p,q}
    \label{M3}
  \end{align}
  are bounded for some $\sigma>s-2$.

  Moreover, in the $F^s_{p,q}$ case, \upn{(ii)} alone implies that
  \eqref{M3} holds for $q=\infty$ and $\sigma$ equal, for any 
  $\varepsilon>0$, to
  \begin{equation}
    \sigma(s,p)=
    \begin{cases}
      s&\text{for $s>\fracnp$ or $0<s<1$,} \\
      s-\varepsilon&\text{for $s=\fracnp$ or $s=1$,} \\
      \tfrac{\fracnp}{\fracnp-s+1}&\text{otherwise}.
    \end{cases}
    \label{sgm-cnd}
  \end{equation}
For $E^s_{p,q}$ with $q\in \,]0,\infty]$ it is possible to take 
$\sigma=\sigma(s,p)-\varepsilon$, for any $\varepsilon>0$. 

  When \upn{(i)}--\upn{(iii)} hold, we say that $(s,p,q)$ belongs to
  $\Dm(A_T+g(\cdot))$. 
\end{thm}

This theorem gives sufficent conditions for $g(\cdot)$ to be of a
lower order than $A_T$, so it may be termed the Direct Regularity
Theorem for \eqref{cmp-pb}.
\linebreak[5]

In comparison with \eqref{M1}, we have excluded borderline cases with
equality in (i) and values of $s$ between $\fracnp-n$ and
$\fracnp-\frac{p}{1-p}$. The latter restriction is
felt in a small set of $(s,p,q)$'s, for in (ii) it only applies
for $p<1$ and in this region $s>r+\fracnp-n$ is stronger to begin with
(since $r=1$ or $2$) and afterwards the second requirement in (iii) quickly
takes over, cf.~Figure~1. The first part of (iii) is 
stronger than $s>\fracnp-\frac{p}{1-p}$, hence stronger than (ii). Exceptions
for $n=1$, $2$, $3$ or $r=2$ are given in
Remarks~\ref{r=2-rem}--\ref{condition-rem} below.

It is expected, but not proved, that the function $\sigma(s,p)$ in
\eqref{sgm-cnd} may be 
used in \eqref{M3} also for $q<\infty$, and even then also in the
Besov case. 

Nevertheless the function $\sigma(s,p)$ gives the right understanding
of the conditions (ii)--(iii) (the sum-exponents are
less important because $E^s_{p,q}\hookrightarrow
E^s_{p,r}$ for $q\le r$). On the one hand, (ii) gives either
$s>(\fracnp-n)_+$, so that $E^s_{p,q}\subset L_1^{\op{loc}}$ and hence
$g(\cdot)$ makes sense, or $s>\fracnp-\frac{p}{1-p}$, which may be
seen to yield $E^{\sigma(s,p)}_{p,q}\subset L_1^{\op{loc}}$. Perhaps
the latter condition is only proof-technical; it is used to make sense of
products $u\dots u$ when estimating $g(u)$.

On the other hand, asking for the identity
\begin{equation}
  \sigma(s,p)=s-2,
  \label{M4}
\end{equation}
or for the level curve for the value $2$ of the loss-of-smoothness
function $s-\sigma(s,p)$, one finds
\begin{equation}
  (2s-\fracnp-3)^2=(\fracnp-3)^2-8,
  \label{M5}
\end{equation}
which leads to (iii) with $=$ instead of the inequalities for~$s$.

In other words: condition (iii), or \eqref{M5},
determines a borderline to a region of spaces where the loss of
smoothness equals or exceeds $2$. Generally speaking this is correct,
for if (iii) is violated by $E^s_{p,q}$ then $u\mapsto
\sin(u)$, for example, cannot map into $E^{s-2+\varepsilon}_{p,q}$ for
any $\varepsilon>0$; cf.~ Remark~\ref{Dahlberg-rem} below.

The identity in \eqref{M5} describes a hyperbola in the
$(\fracnp,s)$-halfplane, that  lies entirely in the area with
$1<s<\fracnp$. Hence (iii) is relevant only for the consideration of
unbounded solutions in \eqref{AT-pb}.

To present an overview, the spaces $E^s_{p,q;T}$ for which the
perturbation $g(u)$ is studied in the present article are illustrated
in Figure~\ref{DM-fig} (for simplicity only for $r=1$). 
The sum-exponent $q$ is not represented in the diagram, but because of
the sharp inequalities in Theorem~\ref{dreg-thm} and the existence of simple
embeddings, $q$ does not have any influence.

\begin{figure}[htbp]
\hfil
\setlength{\unitlength}{0.0125in}
\begin{picture}(380,360)(0,0)

\path(7,30)(13,30)                                   
\path(7,70)(13,70)                                   
\path(126.569,7)(126.569,13)                         
\dottedline{5}(130,20)(130,220)

\thicklines
\path(10,10)(360,39.1667)                            
\path(230,10)(310,90)                                

\path(10,0)(10,350)\path(12,342)(10,350)(8,342)
\path(8,10)(370,10)\path(362,8)(370,10)(362,12)
\put(7,355){\makebox(0,0)[lb]{$s$}}
\put(375,9){\makebox(0,0)[lc]{$\fracnp$}}

\put(0,7){\makebox(0,0)[lb]{$\scriptstyle 0$}}
\put(5,30){\makebox(0,0)[rc]{$\scriptstyle 1$}}
\put(5,70){\makebox(0,0)[rc]{$\scriptstyle 3$}}
\put(126.568,5){\makebox(0,0)[ct]{$\scriptscriptstyle 3+\sqrt8$}}

\put(130,300){\makebox(0,0)[cb]{$\Dm(A_T+g(\cdot))$}}
\put(125,200){\makebox(0,0)[rb]{$\scriptstyle p=2$}}
\put(350,40){\makebox(0,0)[cb]{(i) $\scriptstyle s=\fracpi-1+r$}}
\put(350,95){\makebox(0,0)[cb]{(i) $\scriptstyle s=\fracci np-n+r$}}
\put(185,150){\makebox(0,0)[cb]{(iii)}}
\put(350,210){\makebox(0,0)[cb]{(ii) $\scriptstyle s=\fracci np-\frac p{1-p}$}}


\path(126.569,98.284)(127.069,102.303)(127.569,104.126)
(128.069,105.591)(128.569,106.872)(129.069,108.036)(129.569,109.117)
(130.569,111.107)(131.569,112.936)(132.569,114.652)(133.569,116.284)
(134.569,117.849)(135.569,119.362)(136.569,120.830)(137.569,122.261)
(138.569,123.660)(139.569,125.031)(140.569,126.379)(141.569,127.705)
(142.569,129.012)(143.569,130.302)(144.569,131.577)(145.569,132.837)
(146.569,134.085)(147.569,135.321)(148.569,136.547)(149.569,137.763)
(150.569,138.969)(151.569,140.167)(152.569,141.357)(153.569,142.540)
(154.569,143.716)(155.569,144.886)(156.569,146.049)(157.569,147.207)
(158.569,148.359)(159.569,149.506)(160.569,150.649)(161.569,151.787)
(162.569,152.921)(163.569,154.050)(164.569,155.176)(165.569,156.298)
(166.569,157.417)(167.569,158.532)(168.569,159.644)(169.569,160.753)
(170.569,161.860)(171.569,162.963)(172.569,164.064)(173.569,165.162)
(174.569,166.258)(175.569,167.351)(176.569,168.442)(177.569,169.531)
(178.569,170.618)(179.569,171.702)(180.569,172.785)(181.569,173.866)
(182.569,174.946)(183.569,176.023)(184.569,177.099)(185.569,178.173)
(186.569,179.246)(187.569,180.317)(188.569,181.386)(189.569,182.455)
(190.569,183.521)(191.569,184.587)(192.569,185.651)(193.569,186.714)
(194.569,187.776)(195.569,188.837)(196.569,189.896)(197.569,190.954)
(198.569,192.012)(199.569,193.068)(200.569,194.123)(201.569,195.178)
(202.569,196.231)(203.569,197.283)(204.569,198.335)(205.569,199.385)
(206.569,200.435)(207.569,201.484)(208.569,202.532)(209.569,203.580)
(210.569,204.626)(211.569,205.672)(212.569,206.717)(213.569,207.761)
(214.569,208.805)(215.569,209.848)(216.569,210.890)(217.569,211.932)
(218.569,212.973)(219.569,214.014)(220.569,215.053)(221.569,216.093)
(222.569,217.131)(223.569,218.169)(224.569,219.207)(225.569,220.244)
(226.569,221.280)(227.569,222.316)(228.569,223.352)(229.569,224.387)
(230.569,225.421)(231.569,226.455)(232.569,227.489)(233.569,228.522) 
(234.569,229.555)(235.569,230.587)(236.569,231.619)(237.569,232.650)
(238.569,233.681)(239.569,234.712)(240.569,235.742)(241.569,236.772)
(242.569,237.801)(243.569,238.830)(244.569,239.859)(245.569,240.887)
(246.569,241.915)(247.569,242.943)(248.569,243.970)(249.569,244.997)
(250.569,246.024)(251.569,247.050)(252.569,248.076)(253.569,249.102)
(254.569,250.127)(255.569,251.152)(256.569,252.177)(257.569,253.202)
(258.569,254.226)(259.569,255.250)(260.569,256.274)(261.569,257.297)
(262.569,258.320)(263.569,259.343)(264.569,260.366)(265.569,261.389)
(266.569,262.411)(267.569,263.433)(268.569,264.454)(269.569,265.476)
(270.569,266.497)(271.569,267.518)(272.569,268.539)(273.569,269.560)
(274.569,270.580)(275.569,271.600)(276.569,272.620)(277.569,273.640)
(278.569,274.660)(279.569,275.679)(280.569,276.698)(281.569,277.717)
(282.569,278.736)(283.569,279.755)(284.569,280.773)(285.569,281.791)
(286.569,282.809)(287.569,283.827)(288.569,284.845)(289.569,285.862)
(290.569,286.880)(291.569,287.897)(292.569,288.914)(293.569,289.931)
(294.569,290.948)(295.569,291.964)(296.569,292.981)(297.569,293.997)
(298.569,295.013)(299.569,296.029)(300.569,297.045)(301.569,298.061)
(302.569,299.076)(303.569,300.092)(304.569,301.107)(305.569,302.122)
(306.569,303.137)(307.569,304.152)(308.569,305.167)(309.569,306.181)
(310.569,307.196)(311.569,308.210)(312.569,309.224)(313.569,310.239)
(314.569,311.253)(315.569,312.266)(316.569,313.280)(317.569,314.294)
(318.569,315.307)(319.569,316.321)(320.569,317.334)(321.569,318.347)
(322.569,319.360)(323.569,320.373)(324.569,321.386)(325.569,322.399)
(326.569,323.412)(327.569,324.424)(328.569,325.437)(329.569,326.449)
(330.569,327.461)(331.569,328.473)(332.569,329.486)(333.569,330.497)
(334.569,331.509)(335.569,332.521)(336.569,333.533)(337.569,334.544)
(338.569,335.556)(339.569,336.567)(340.569,337.579)(341.569,338.590)
(342.569,339.601)(343.569,340.612)(344.569,341.623)(345.569,342.634)
(346.569,343.645)(347.569,344.656)(348.569,345.666)(349.569,346.677)
(350.569,347.688)(351.569,348.698)(352.569,349.708)(353.569,350.719)
(354.569,351.729)(355.569,352.739)(356.569,353.749)(357.569,354.759)
(358.569,355.769)(359.569,356.779)


\path(126.569,98.284)(127.069,94.765)(127.569,93.443)
(128.069,92.478)(128.569,91.697)(129.569,90.451) 
(130.569,89.461)(131.569,88.632)(132.569,87.916)(133.569,87.285)
(134.569,86.719)(135.569,86.207)(136.569,85.739)(137.569,85.308)
(138.569,84.909)(139.569,84.537)(140.569,84.190)(141.569,83.864)
(142.569,83.557)(143.569,83.267)(144.569,82.992)(145.569,82.731)
(146.569,82.483)(147.569,82.247)(148.569,82.022)(149.569,81.806)
(150.569,81.599)(151.569,81.401)(152.569,81.211)(153.569,81.028)
(154.569,80.852)(155.569,80.683)(156.569,80.520)(157.569,80.362)
(158.569,80.209)(159.569,80.062)(160.569,79.920)(161.569,79.782)
(162.569,79.648)(163.569,79.518)(164.569,79.392)(165.569,79.270)
(166.569,79.152)(167.569,79.036)(168.569,78.924)(169.569,78.815)
(170.569,78.709)(171.569,78.606)(172.569,78.505)(173.569,78.407)
(174.569,78.311)(175.569,78.218)(176.569,78.127)(177.569,78.038)
(178.569,77.951)(179.569,77.866)(180.569,77.783)(181.569,77.702)
(182.569,77.623)(183.569,77.546)(184.569,77.470)(185.569,77.396)
(186.569,77.323)(187.569,77.252)(188.569,77.182)(189.569,77.114)
(190.569,77.047)(191.569,76.982)(192.569,76.917)(193.569,76.854)
(194.569,76.793)(195.569,76.732)(196.569,76.672)(197.569,76.614)
(198.569,76.557)(199.569,76.500)(200.569,76.445)(201.569,76.391)
(202.569,76.338)(203.569,76.285)(204.569,76.234)(205.569,76.183)
(206.569,76.133)(207.569,76.084)(208.569,76.036)(209.569,75.989)
(210.569,75.942)(211.569,75.897)(212.569,75.852)(213.569,75.807)
(214.569,75.763)(215.569,75.720)(216.569,75.678)(217.569,75.637)
(218.569,75.595)(219.569,75.555)(220.569,75.515)(221.569,75.476)
(222.569,75.437)(223.569,75.399)(224.569,75.362)(225.569,75.325)
(226.569,75.288)(227.569,75.252)(228.569,75.217)(229.569,75.182)
(230.569,75.147)(231.569,75.113)(232.569,75.080)(233.569,75.047)
(234.569,75.014)(235.569,74.982)(236.569,74.950)(237.569,74.919)
(238.569,74.888)(239.569,74.857)(240.569,74.827)(241.569,74.797)
(242.569,74.768)(243.569,74.738)(244.569,74.710)(245.569,74.681)
(246.569,74.653)(247.569,74.626)(248.569,74.598)(249.569,74.572)
(250.569,74.545)(251.569,74.518)(252.569,74.492)(253.569,74.467)
(254.569,74.441)(255.569,74.416)(256.569,74.391)(257.569,74.367)
(258.569,74.342)(259.569,74.318)(260.569,74.295)(261.569,74.271)
(262.569,74.248)(263.569,74.225)(264.569,74.202)(265.569,74.180)
(266.569,74.158)(267.569,74.136)(268.569,74.114)(269.569,74.093)
(270.569,74.071)(271.569,74.050)(272.569,74.029)(273.569,74.009)
(274.569,73.988)(275.569,73.968)(276.569,73.948)(277.569,73.929)
(278.569,73.909)(279.569,73.890)(280.569,73.870)(281.569,73.851)
(282.569,73.833)(283.569,73.814)(284.569,73.796)(285.569,73.777)
(286.569,73.759)(287.569,73.741)(288.569,73.724)(289.569,73.706)
(290.569,73.689)(291.569,73.671)(292.569,73.654)(293.569,73.638)
(294.569,73.621)(295.569,73.604)(296.569,73.588)(297.569,73.571)
(298.569,73.555)(299.569,73.539)(300.569,73.524)(301.569,73.508)
(302.569,73.492)(303.569,73.477)(304.569,73.462)(305.569,73.446)
(306.569,73.431)(307.569,73.417)(308.569,73.402)(309.569,73.387)
(310.569,73.373)(311.569,73.358)(312.569,73.344)(313.569,73.330)
(314.569,73.316)(315.569,73.302)(316.569,73.288)(317.569,73.275)
(318.569,73.261)(319.569,73.248)(320.569,73.234)(321.569,73.221)
(322.569,73.208)(323.569,73.195)(324.569,73.182)(325.569,73.170)
(326.569,73.157)(327.569,73.144)(328.569,73.132)(329.569,73.120)
(330.569,73.107)(331.569,73.095)(332.569,73.083)(333.569,73.071)
(334.569,73.059)(335.569,73.047)(336.569,73.036)(337.569,73.024)
(338.569,73.013)(339.569,73.001)(340.569,72.990)(341.569,72.979)
(342.569,72.967)(343.569,72.956)(344.569,72.945)(345.569,72.934)
(346.569,72.923)(347.569,72.913)(348.569,72.902)(349.569,72.891)
(350.569,72.881)(351.569,72.870)(352.569,72.860)(353.569,72.850)
(354.569,72.840)(355.569,72.829)(356.569,72.819)(357.569,72.809)
(358.569,72.799)(359.569,72.790)


\path(269,16.368)
(270,30.000)(271,42.429)(272,53.818)(273,64.304)(274,74.000)
(275,83.000)(276,91.385)(277,99.222)(278,106.571)(279,113.483)
(280,120.000)(281,126.161)(282,132.000)(283,137.545)(284,142.824)
(285,147.857)(286,152.667)(287,157.270)(288,161.684)(289,165.923)
(290,170.000)(291,173.927)(292,177.714)(293,181.372)(294,184.909)
(295,188.333)(296,191.652)(297,194.872)(298,198.000)(299,201.041)
(300,204.000)(301,206.882)(302,209.692)(303,212.434)(304,215.111)
(305,217.727)(306,220.286)(307,222.789)(308,225.241)(309,227.644)
(310,230.000)(311,232.311)(312,234.581)(313,236.810)(314,239.000)
(315,241.154)(316,243.273)(317,245.358)(318,247.412)(319,249.435)
(320,251.429)(321,253.394)(322,255.333)(323,257.247)(324,259.135)
(325,261.000)(326,262.842)(327,264.662)(328,266.462)(329,268.241)
(330,270.000)(331,271.741)(332,273.463)(333,275.169)(334,276.857)
(335,278.529)(336,280.186)(337,281.828)(338,283.455)(339,285.067)
(340,286.667)(341,288.253)(342,289.826)(343,291.387)(344,292.936)
(345,294.474)(346,296.000)(347,297.515)(348,299.020)(349,300.515)
(350,302.000)(351,303.475)(352,304.941)(353,306.398)(354,307.846)
(355,309.286)(356,310.717)(357,312.140)(358,313.556)(359,314.963)

\end{picture}
\hfil
\caption{The set $\Dm(A_T+g(\cdot))$ for $n=12$ and $r=1$; the boundary
curves are labelled as in Theorem~\ref{dreg-thm}. Dots
indicate the spaces with $p=2$.}
  \label{DM-fig}
\end{figure}
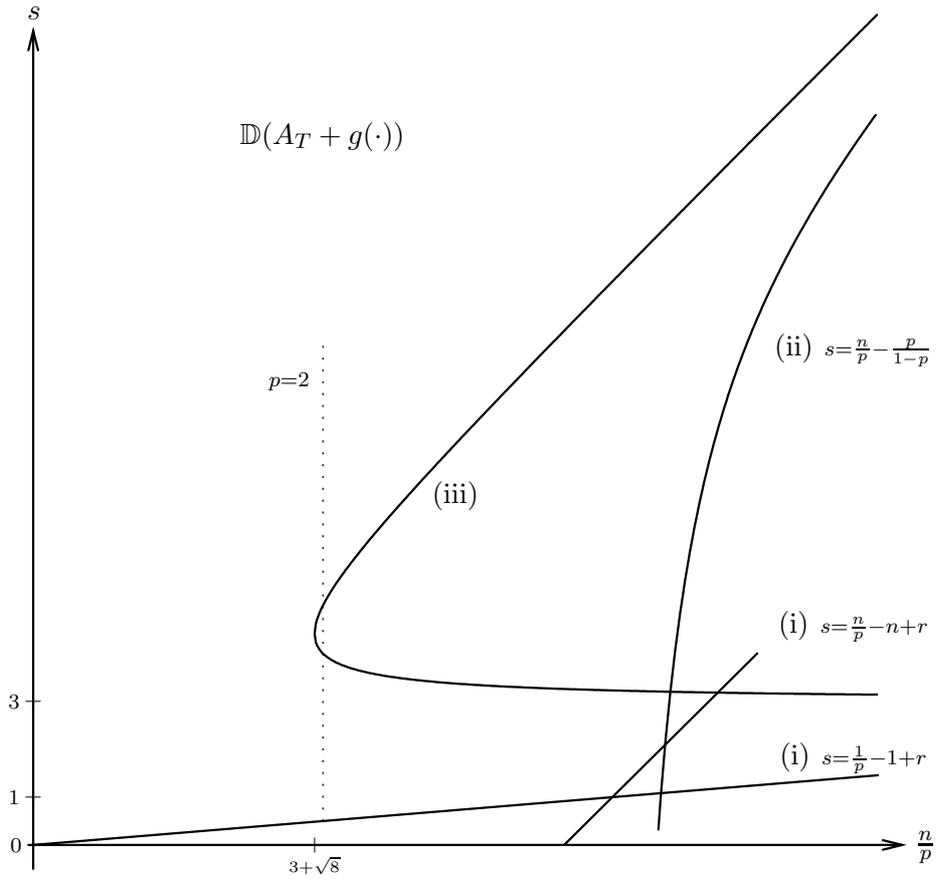

The lines with $s=3$ and $s=\fracnp$ are the asymptotes of the
hyperbola, and for all points on this level curve,
\begin{equation}
  \fracnp\ge 3+\sqrt8.
  \label{M6}
\end{equation}
The interest of this is that for $n\ge12$ even the theory within the
classical $H^s$ Sobolev spaces is affected by (iii) in
Theorem~\ref{dreg-thm}. Actually $s$ should be taken outside of an
interval of length $\sqrt{(\frac{n}{2}-3)^2-8}$, which is at least $1$
and $\cal O(\frac{n}{2})$ for $n\to \infty$. Moreover, for each
$n\ge6$ there are $p>1$ fulfilling \eqref{M6}, so
restrictions occur also in the $W^s_p$ and
$H^s_p$ spaces for such dimensions. 

In addition to the general pattern described above, see
Section~\ref{inter-ssect} below for the atypical cases with $n=1$, $n=2$ or
$r=2$. 

At the moment it is not clear whether the condition
$s>\fracnp-\frac{p}{1-p}$ is necessary or not, but in any case it
won't change the fact that the sets $\Dm(A_T+g(\cdot))$ are non-convex, 
because already for $g(u)=\sin(u)$ the condition (iii)
is best possible. We believe that the specific form of the $\Dm$'s and
in particular the non-convexity constitutes a novelty.

\bigskip

Because $\sigma>s-2$ is possible in $\Dm(A_T+g(\cdot))$, 
the non-linear operator
$g(\cdot)$ also respects the \emph{inverse} regularity properties of $A_T$ on
every $E^s_{p,q;T}$ with parameter in $\Dm(A_T+g(\cdot))$:

\begin{thm} \label{ireg-thm}
  Let  $u(x)$ in $E^{s_1}_{p_1,q_1;T}$ solve
  \begin{equation}
    A_Tu+g(u)=f
    \label{M7}
  \end{equation}
  for data $f(x)$ in $E^{s_0-2}_{p_0,q_0}$ and suppose that
  \begin{equation}
    (s_1,p_1,q_1), (s_0,p_0,q_0)\in\Dm(A_T+g(\cdot)).
    \label{M8}
  \end{equation}
  Then the solution $u(x)$ also belongs to the space
  $E^{s_0}_{p_0,q_0;T}$. 
\end{thm}
To prove this we use Theorem~\ref{dreg-thm}
for $g(u)$ and results for the Boutet de~Monvel calculus in
\cite{JJ96ell} for $A_T$. These tools are combined into a bootstrap
argument, but one has to `go around the corner' inside
$\Dm(A_T+g(\cdot))$, because of the non-convexity; cf.\ 
Figure~\ref{wrst-fig} below. 

It is interesting to observe that the set $\Dm(A_T+g(\cdot))$\,---\,in
contrast to Theorem~\ref{dreg-thm}\,---\,is non-optimal with respect to
$(s_0,p_0,q_0)$, cf.\ Remark~\ref{ireg-rem}.

\bigskip

Concerning the solvability of the problem in \eqref{AT-pb}  
it is noted that
the Fredholm properties of $A_T$ depend neither on the parameter
$(s,p,q)$ nor on whether the $B^s_{p,q}$ or the $F^s_{p,q}$ spaces are
considered.

That is to say, because of the ellipticity and the right-invertibility
of $T$, there exists two finite dimensional subspaces $\ker{A_T}$ and $\cal N$
of $C^\infty(\overline{\Omega})$ such that when
$s>r+\max(\fracp-1,\fracnp-n)$ the following holds:
\begin{align}
  \ker{A_T}&=\bigl\{\,u\in E^s_{p,q;T}\bigm| A_Tu=0\,\bigr\},
  \label{M11} \\
  E^{s-2}_{p,q}&=\cal N\oplus A_T(E^s_{p,q;T});  
  \label{M12}
\end{align}
and $A_T(E^s_{p,q;T})$ is closed. This is a consequence of
\cite[Thm.~1.3]{JJ96ell}; see Section~\ref{ireg-sect} below for details.
In particular $A_T$ is bijective for all admissible parameters $(s,p,q)$ if
(and only if) it is so for one.

Among the conditions that assure solvability of \eqref{AT-pb} we consider:
\begin{itemize}
  \item[(I)] $A_T$ is invertible.
  \label{inv-cnd}
  \item[(II)] For each bounded sequence $(v_k)$ in $L_{t-0}$,
$\frac1t=(\fracp-\frac sn)_+$, and each
$L_\infty$-convergent sequence $(w_k)$ in $\ker A_T$ with
$\norm{w_k}{L_\infty}=1$, 
\begin{equation}
  \int_\Omega g(v_k+t_kw_k)w_k\,dx-\dual{f}{w_k}\ge0
  \label{RL+-cnd}
\end{equation}
holds for some $k\in\N$ when $t_k\to \infty$ for $k\to \infty$.
  \item[(III)] Under the hypothesis of (II),
\begin{equation}
  \int_\Omega g(v_k+t_kw_k)w_k\,dx-\dual{f}{w_k}\le0
  \label{RL--cnd}
\end{equation}
holds for some $k\in\N$ when $t_k\to \infty$ for $k\to \infty$.
\end{itemize}
It should be understood that $L_{t-0}$ means $L_t$, except when
$B^{s}_{p,q;T}$ is considered for $q>t$ where $t-0$ denotes any $t'<t$. This
ensures $E^{s}_{p,q;T}\hookrightarrow L_{t-0}$ in any case, cf.\
\eqref{1.22}--\eqref{1.25}.

Both (II) and (III) are posed for each $f$ in $E^{s-2}_{p,q}$ with
$(s,p,q)$ in $\Dm(A_T+g(\cdot))$; since the requirements are void if $A_T$
is injective, (I) implies both of them. When $g(t)$ is odd, 
$\text{(II)}\iff\text{(III)}$ holds, reflecting that $A_T+g(\cdot)$ then
sends $u$ to $f$ if and only if $-u$ is mapped to $-f$. If $g$ is even, then
(II) holds for $f$ precisely when $-f$ satisfies (III) for $-g$ (and
$A_Tu+g(u)=f$ if and only if $A_T-g(\cdot)$ maps $-u$ to $-f$, then). 

\begin{thm} \label{solv-thm}
  Let  $(s,p,q)$ fulfil 
  \upn{(i)}--\upn{(iii)} in Theorem~\ref{dreg-thm}, let $f(x)$
  be given in $E^{s-2}_{p,q}$, and let $A_T$
  satisfy \upn{(I)}, or let $A_T$ be self-adjoint and $f(x)$  have
  one of the properties in \upn{(II)} or \upn{(III)} above. 
   Then the equation
  \begin{equation}
    A_Tu+g(u)=f
    \label{M10}
  \end{equation}
  has at least one solution $u(x)$ belonging to $E^s_{p,q;T}$.
\end{thm}
This generalises the $L_2$-versions of (III) of Robinson and
Landesman \cite{RoLa95} and the $B^{s}_{p,q}$- and $F^{s}_{p,q}$-version of
(II) in \cite{RoRu96} to the case with 
$(s,p,q)$ in the full parameter domain $\Dm(A_T+(\cdot))$ as defined here. See
Remarks~\ref{orignorm-rem}--\ref{RoRu-rem} below for specific comparisons.

Simple cases of Theorem~\ref{solv-thm} are given in
Examples~\ref{data-ex}--\ref{ireg-ex} above. In addition, note that we can
have, say, $\mlap_{\gamma_0}-\lambda$ where $\lambda$ is any eigenvalue.

One-dimensional examples may be
found in e.g.~\cite{RoLa95}; they also elucidate the
connection to other and earlier conditions, mainly formulated in terms of
$g(t)$'s properties and without reference to sequences. For the $B^{s}_{p,q}$
and $F^{s}_{p,q}$ conditions there is a similar treatment in
\cite{RoRu96}. Drawing on this, we do not give further examples on (II) and
(III). 

Concerning the proof we use when $s<2$ that 
$L_\infty(\Omega)\hookrightarrow E^{s-2}_{p,q}$ to obtain
Theorem~\ref{solv-thm} from the Leray--Schauder theorem. 
The remaining cases are reduced to this by a crucial
application of Theorem~\ref{ireg-thm}, cf.\ Section~\ref{ex-sect}.

\begin{rem}   \label{orignorm-rem}
In (II) and (III) it suffices when $s<2$ and $1<p\le\infty$ to consider
sequences $(v_k)$ that are merely bounded in $E^{s}_{p,q;T}$ \emph{itself}. Our
proof gives this directly, but the $L_{t-0}$-condition is convenient to state.
\end{rem}

\begin{rem}
  \label{RoLa-rem}
Formally the requirements in (II) and (III) are weaker than those in e.g.\
\cite{RoLa95} in the sense that the inequalities should hold for one $k$ in
$\N$, and not for all $k$ eventually. However, it is easy to infer that this
must be the case when (II) or (III) holds.
 
Seemingly (II) and (III) have not been considered simultaneously before.
\end{rem}
\begin{rem}
  \label{RoRu-rem}
Extension to $B^{s}_{p,q}$ and $F^{s}_{p,q}$ of the conditions in
\cite{RoLa95} has been done by Robinson and Runst \cite{RoRu96}, but only for
$s>\fracnp$. Conditions (II) and (III) are also more general in other
respects. Most importantly,
we have removed the additional assumption that $f\in B^{t}_{\infty,\infty}$
for $t>-1$ when $T$ has class $2$. Secondly, (II) and (III) may by
Remark~\ref{orignorm-rem} in some cases refer to the
$E^s_{p,q}$-norms (implying their $L_\infty$-conditions when $s>\fracnp$); 
thirdly $(v_k)$ is assumed bounded, so that it is unnecessary to consider the
case when their norms tend slower to infinity than $(t_k)$. 
\end{rem}

\subsection{Notation}  \label{notation-ssect}
For real numbers $a$ the convention $a_\pm=\max(0,\pm a)$ is used. 
When $A\subset\Rn$  is open, $L_p(A)$ denotes the classes of 
functions whose $p^{th}$ power is integrable for $0<p<\infty$, while
$p=\infty$ gives the essentially bounded ones; $L_1^{\op{loc}}(A)$
stands for the locally integrable functions.

When $\Omega\subset\Rn$ is open, $C^\infty(\Omega)$ denotes the
infinitely differentiable functions; $\Cb^\infty(\Rn)$ the subspace of
$C^\infty(\Rn)$ for which derivatives of any order are bounded. $\cal
S(\Rn)$ is the Schwartz space of rapidly decreasing functions; $\cal
S'(\Rn)$ its dual of tempered distibutions. The Fourier
transformation $\cal F$ is extended to $\cal S'$ by duality. 
The Sobolev--Slobodetski\u\i\ spaces $W^s_p$ are defined by
derivatives and differences thereof for $s>0$ and $1<p<\infty$;
the Bessel potential spaces $H^s_p=\cal F^{-1}(1+|\xi|^2)^{-s/2}\cal F(L_p)$
for $s\in\R$, $1<p<\infty$. Besov and Triebel--Lizorkin spaces are
written $B^{s}_{p,q}(\Rn)$ and $F^{s}_{p,q}(\Rn)$ with $s\in\R$ while
$p$,~$q\in\,]0,\infty]$, except that $p<\infty$ is required for~$F^{s}_{p,q}$.

The subspaces of \emph{real-valued} elements are all denoted by the same
symbols as the complex ones, for throughout we only consider the former
versions.  

For open sets $\Omega\subset\Rn$ the corresponding spaces are defined
by restriction, that is $B^{s}_{p,q}(\overline{\Omega})=
r_\Omega B^{s}_{p,q}(\Rn)$ etc. Hereby $r_\Omega$ is
the transpose of $e_\Omega$, the extension by $0$ outside of $\Omega$.
Spaces over $\Omega$ are given the infimum (quasi-) norm. Similarly
for $C^\infty(\overline{\Omega})$. For the
testfunction space $C^\infty_0(\Omega)$ the dual is written $\cal
D'(\Omega)$, and $\dual{u}{\varphi}=u(\varphi)$ for $u\in\cal D'$ and
$\varphi\in C_0^\infty$. The spaces over $\Gamma=\partial\Omega$ are
defined by means of local coordinates.

\subsection{The spaces} \label{spaces-ssect}
In the following $\Rn$ is suppressed as the underlying set.

First a partition of unity, $1=\sum_{j=0}^\infty\Phi_j$, is 
constructed: From $\Psi\in C^\infty(\R)$, such
that $\Psi(t)=1$ for $0\le t\le\tfrac{11}{10}$ and $\Psi(t)=0$
for $\tfrac{13}{10}\le t$, the functions
$ \Psi_j(\xi)=\Psi(2^{-j}|\xi|)$, with $\Psi_j\equiv0$ for $j<0$,
are used to define
 \begin{equation}
  \Phi_j(\xi)=\Psi_j(\xi)-\Psi_{j-1}(\xi),
  \quad\text{ for}\quad j\in\Z.
  \label{1.15} 
 \end{equation}
Secondly there is then a decomposition, with (weak) convergence 
in $\cal S'$,
 \begin{equation}
 u=\sum_{j=0}^\infty\,\cal F^{-1}(\Phi_j\cal Fu),
 \quad\text{ for every}\quad u\in\cal S'.
 \label{1.16} 
 \end{equation}
Now the Besov space $B^{s}_{p,q}(\Rn)$ and the Triebel--Lizorkin space
$F^{s}_{p,q}(\Rn)$  with \emph{smoothness 
index $s\in\R$, integral-exponent $p\in\left]0,\infty\right]$ 
{\rm and} sum-exponent} $q\in\left]0,\infty\right]$ is defined as
 \begin{align}
 B^{s}_{p,q}&=\bigl\{\,u\in\cal S'\bigm|
 \Norm{ \{2^{sj} \norm{\cal F^{-1}\Phi_j\cal Fu}{L_p} \}_{j=0}^\infty
       }{\ell_q} <\infty\,\bigr\},
 \label{1.17}
\\ 
 F^{s}_{p,q}&=\bigl\{\,u\in\cal S'\bigm|
 \Norm{ \norm{ \{2^{sj}\cal F^{-1}\Phi_j\cal
 Fu\}_{j=0}^\infty}{\ell_q} (\cdot)}{L_p} <\infty\,\bigr\},
 \label{1.18}
 \end{align}
respectively. For the history of these spaces we refer to Triebel's books
\cite{T2,T3}.  Identifications with other spaces are found in 
Section~\ref{intr-sect}.

In the rest of this subsection the explicit mention of the restriction 
$p<\infty$ concerning the~Triebel--Lizorkin spaces is omitted. E.g., 
\eqref{1.19'} below should be read with $p\in\,]0,\infty]$ in 
the $B^{s}_{p,q}$ part and with $p\in\,]0,\infty[\,$ in 
the $F^{s}_{p,q}$ part. 

The $B^{s}_{p,q}$ and $F^{s}_{p,q}$ are complete, for $p$ and $q\ge1$
Banach spaces, and $\cal S\hookrightarrow E^{s}_{p,q} \hookrightarrow\cal
S'$ are continuous. Moreover, $\cal S$ is dense in $E^{s}_{p,q}$ 
when both $p$ and $q$ are finite, and $C^\infty$ is so in 
$B^{s}_{\infty,q}$ for $q<\infty$. 

The definitions imply that $B^{s}_{p,p}=F^{s}_{p,p}$,
and they give the existence of \emph{simple} embeddings for 
$s\in\R,\,\,p\in\left]0,\infty\right]$ 
and $o$ and $q\in\left]0,\infty\right]$, 
  \begin{gather}
  E^{s}_{p,q}\hookrightarrow E^{s}_{p,o}
                              \quad\text{when $q\le o$}, \qquad
  E^{s}_{p,q}\hookrightarrow E^{s-\varepsilon}_{p,o},
\quad\text{$ \varepsilon>0,$}
   \label{1.19'} \\ 
  B^{s}_{p,\min(p,q)}\hookrightarrow F^{s}_{p,q}
     \hookrightarrow B^{s}_{p,\max(p,q)}.
 \label{1.19''}
 \end{gather}
There are Sobolev embeddings if $s-\fracc{n}p\ge
t-\fracc{n}r$ and $r>p$, more specifically
 \begin{gather}
 B^{s}_{p,q}\hookrightarrow B^{t}_{r,o},
 \quad\text{ provided $q\le o$ when $ s-\fracc{n}p
 =t-\fracc{n}r$},
 \label{1.20} \\
 F^{s}_{p,q}\hookrightarrow F^{t}_{r,o},
 \quad\text{ for any $o$ and $q\in\,]0,\infty].$}
 \label{1.20'}
 \end{gather}
Furthermore, Sobolev embeddings also exist between the two scales,
in fact under the assumptions $\infty\ge p_1>p>p_0>0$ and 
$s_0-\fracc{n}{p_0}=s-\fracc{n}p=s_1-\fracc{n}{p_1}$,
 \begin{equation}
  B^{s_0}_{p_0,q_0}\hookrightarrow F^{s}_{p,q}
  \hookrightarrow B^{s_1}_{p_1,q_1},\quad
  \text{for $q_0\le p$ and $p\le q_1$.}  
  \label{1.21} 
 \end{equation}
When $\Cb$ denotes the bounded uniformly continuous functions on $\Rn$, then
 \begin{equation}
  \begin{gathered}
  B^{s}_{p,q}\hookrightarrow B^{0}_{\infty, 1}\hookrightarrow
  \Cb\hookrightarrow L_\infty\hookrightarrow B^{0}_{\infty,\infty}
  ,\\
  \text{if $s>\fracc{n}p$, or if $s=\fracc{n}p$ and $q\le1$};
  \end{gathered}
  \label{1.22} 
 \end{equation} 
whereas
 \begin{equation}
  \begin{gathered}
  F^{s}_{p,q}\hookrightarrow B^{0}_{\infty,1}\hookrightarrow
  \Cb\hookrightarrow L_\infty,\\
  \text{if $s>\fracc{n}p$, or if $s=\fracc{n}p$ and $p\le1$}.
  \end{gathered}
  \label{1.23}
 \end{equation}
Moreover, when $n(\fracc{1}p-1)_+\le s<\fracc{n}p$
one has, with $\tfrac{n}{t}=\fracc{n}p-s$, that
 \begin{equation}
  F^{s}_{p,q}\hookrightarrow\bigcap\{\, L_r\mid
  p\le r\le t\,\};
 \label{1.24}
 \end{equation}
for $s=0$ this is provided that $q\le 1$ for $p=1$ and that $q\le
2$ for $p>1$. Correspondingly
 \begin{equation}
 B^{s}_{p,q}\hookrightarrow\bigcap\{\, L_r\mid
 p\le r< t\,\},
 \label{1.25}
 \end{equation}
where $r=t$ can be included in general when $q\le t$. For $s=0$ one has
$B^{s}_{p,q}\hookrightarrow L_p$ for $q\le\min(2,p)$ and $p\ge1$.

\bigskip

For an open set $\Omega\subset\Rn$ the space 
$E^{s}_{p,q}(\overline{\Omega})$ is defined by restriction,
 \begin{gather}
 E^{s}_{p,q}(\overline{\Omega})=r_\Omega E^{s}_{p,q}=
 \{\,u\in\cal D'(\Omega)\mid \exists v\in E^{s}_{p,q}\colon 
 r_\Omega v=u\,\} 
 \label{1.34} \\
 \norm{u}{E^{s}_{p,q}(\overline{\Omega})}=\inf \bigl\{\,
 \norm{v}{E^{s}_{p,q}}\bigm| r_\Omega v=u\,\bigr\}.
 \label{1.34'}
 \end{gather}
By the definitions all the~embeddings in \eqref{1.19'}--\eqref{1.25} carry
over to the scales over $\Omega$. When $\infty\ge p\ge r>0$ the inclusion
$L_p(\Omega)\hookrightarrow L_r(\Omega)$ gives
 \begin{equation}
 B^{s}_{p,q}(\overline{\Omega})\hookrightarrow 
 B^{s}_{r,q}(\overline{\Omega}), \qquad
 F^{s}_{p,q}(\overline{\Omega})\hookrightarrow 
 F^{s}_{r,q}(\overline{\Omega}),
 \label{1.38}
 \end{equation} 
for $\Omega$, say smooth and bounded; cf.~\cite{JJ94mlt} for a proof (in
full generality). 

\begin{prop}  \label{tnsr-prop}
For $s<0$ and $p,q\in\,]0,\infty]$ there exists $c<\infty$ such that 
\begin{align}
  \norm{u\otimes v}{B^{s}_{p,q}(\R^{n+m})}&\le
  c\norm{u}{B^{s}_{p,q}(\Rn)}\norm{v}{L_p(\R^m)}, 
  \label{tnsr-1}
 \\
  \norm{u\otimes v}{B^{s+t}_{p,q}(\R^{n+m})}&\le
  c\norm{u}{B^{s}_{p,q_0}(\Rn)}\norm{v}{B^{t}_{p,q_1}(\R^m)},
  \label{tnsr-2}
\end{align}
when $p>1$ in \eqref{tnsr-1} and $t<0$ and
$\fracc1q=\fracc1{q_0}+\fracc1{q_1}$ in \eqref{tnsr-2}, respectively. 
\end{prop}
\begin{proof}
Using Littlewood--Paley decompositions, this may be proved in the same manner
as \cite[Prop.~2.5]{JJ96ell} (where $v=\delta_0$ was treated).
\end{proof}

\begin{exmp}  \label{tnsr1-ex}
Precisely when $1<p\le\infty$ does 
\begin{equation}
  \op{pv}(\tfrac{1}{x})\in B^{\fracpi-1}_{p,\infty}(\R).
  \label{tnsr-3}
\end{equation}
Indeed, since $\op{pv}(\tfrac{1}{x})=i\cal FH-i\pi\delta_0$, where $H$ is the
Heaviside function it suffices to consider $i\cal F H$. Since $H$ is
homogeneous of degree $0$, $\cal FH$ is in $B^{\fracpi-1}_{p,q}$ if and only
if $\cal F^{-1}(\Phi_0H(-\cdot))$ is in $L_p$. But since 
\begin{equation}
  -x\cal F^{-1}(\Phi_0H(-\cdot))-\tfrac{i}{2\pi}=
  \cal F^{-1}(H(-\cdot)D_\xi\Phi_0)\in L_\infty(\R),
  \label{tnsr-4}
\end{equation}
and $\cal F^{-1}(\Phi_0H(-\cdot))$ is in $\Cb(\R)$, it is in $L_p$ for
$1<p\le\infty$.  
\end{exmp}

\begin{exmp}  \label{tnsr2-ex}
By the proposition and Example~\ref{tnsr1-ex}, with $x=(x',x_n)$ in $\Rn$ for
$n\ge2$, one has for $1<p\le\infty$
\begin{equation}
  r_\Omega(1(x')\otimes\op{pv}\tfrac{1}{x_n}) \in
B^{\fracpi-1}_{p,\infty}(\overline{\Omega}),
  \label{tnsr-5}
\end{equation}
for tensoring instead with $1_B$, the characteristic function of a 
bounded set with $\Omega\subset B\times\R$, which is in $L_p(\R^{n-1})$, 
yields the same restriction to $\Omega$. 
\end{exmp}

\section{Composition Estimates}   \label{cmp-sect}
Here we prove Theorem~\ref{dreg-thm} and substantiate the remarks made after
it. 

\subsection{Proof of Theorem~\ref{dreg-thm}}   \label{pf-ssect}
That $T$ is bounded as in \eqref{M2} when
(i) holds is well known. Concerning the standard traces $\gamma_0$ and
$\gamma_1$ one can consult \cite[Thm.~3.3.3]{T2}, and in general this
is combined with the fact that $S_0$ and $S_1$ has order $d$ and
$d-1$, respectively, in both $B^{s}_{p,q}(\Gamma)$ and $F^{s}_{p,q}(\Gamma)$.

Secondly, it suffices to show \eqref{sgm-cnd} for $g(\cdot)$, for the
fact in \eqref{M3} that $E^s_{p,q}$ is sent into $E^\sigma_{p,q}$ for
some $\sigma>s-2$ is a consequence of this. Indeed, given the property
in \eqref{sgm-cnd} it follows at once that \eqref{M3} holds if
$s>\fracnp$ or if $0<s<1$ does so: for any $\varepsilon>0$ one
can take $\sigma=s-\varepsilon$ and use embeddings, e.g. 
\begin{equation}
  B^{s}_{p,q}(\overline{\Omega})\hookrightarrow 
  F^{s-\frac{\varepsilon}{k}}_{p,\infty}(\overline{\Omega})
  \xrightarrow{g(\cdot)}
  F^{s-\frac{\varepsilon}{k}}_{p,\infty}(\overline{\Omega})  
  \hookrightarrow B^{s-\varepsilon}_{p,q}(\overline{\Omega})
  \label{C1}
\end{equation}
when $k$ is so big that $s-\frac{\varepsilon}{k}>\fracnp$
and $s-\frac{\varepsilon}{k}>\max(0,\fracnp-n,\fracnp-\frac{p}{1-p})$.
For $s=1$, or in the $F$-case even for $s=\fracnp$, a 
similar argument applies.

For $1<s<\fracnp$ we consider for $p$ fixed $s-\sigma(s,p)$, that is
\begin{equation}
  d(s)=s-\frac{\fracnp}{\fracnp-s+1}=\frac{(s-1)(\fracnp-s)}{\fracnp-s+1},
  \label{C2} 
\end{equation}
which measures the loss of smoothness under $g(\cdot)$. (There exists for
$\varepsilon>0$ a $u_\varepsilon\in 
E^s_{p,q}$ such that $g(u_\varepsilon)\notin E^{\sigma(s,p)+
\varepsilon}_{p,q}$, cf.~Remark~\ref{Dahlberg-rem}.)  Since 
\begin{equation}
  d(s)=2 \iff s^2-(\fracnp+3)s+3\fracnp+2=0,
  \label{C3}
\end{equation}
where the discriminant $D=(\fracnp-3)^2 -8$, it is found that $d(s)<2$ holds
\begin{align}
  \text{if}\quad s&>\tfrac{1}{2}(\fracnp+3+\sqrt{(\fracnp-3)^2-8}\,)
  \label{C4}  \\
  \text{or if}\quad s&<\tfrac{1}{2}(\fracnp+3-\sqrt{(\fracnp-3)^2-8}\,);
  \label{C5}
\end{align}
this is condition (iii) in the theorem, for $D\ge0$ holds when
$\fracnp\ge3+\sqrt 8$. Observe that
$(\sqrt{\fracnp}-1)^2=\max\{\,d(s)\mid 1<s<\fracnp\,\}$, and that this
equals $2$ for $\fracnp=3+\sqrt8$ since $D=0$ then. If
$\fracnp<3+\sqrt8$, then $(\sqrt{\fracnp}-1)^2 <2$.

For a given $(s,p,q)$ with $1<s<\fracnp$ and (iii) satisfied we can
now take $\varepsilon>0$ so that $\sigma(s,p)-\varepsilon>s-2$ and
obtain 
\begin{equation}
  F^{s}_{p,q}(\overline{\Omega})\xrightarrow{g(\cdot)} 
  F^{\sigma(s,p)}_{p,\infty}(\overline{\Omega})\hookrightarrow
  F^{\sigma(s,p)-\varepsilon}_{p,q}(\overline{\Omega}),
  \label{C6}
\end{equation}
which gives \eqref{M3} in this case. Moreover, the fact that
(ii),(iii) and $1<s<\fracnp$ specify an open set of parameters
$(s,p,q)$ together with the continuity of $\sigma(\cdot,p)$ gives an
$\eta>0$ such that $\sigma(s-\eta,p)>s-2$, and then
\begin{equation}
  B^{s}_{p,q}(\overline{\Omega})\hookrightarrow
  F^{s-\eta}_{p,\infty}(\overline{\Omega})\xrightarrow{g(\cdot)}
  F^{\sigma(s-\eta,p)}_{p,\infty}(\overline{\Omega})
  \hookrightarrow B^{\sigma}_{p,q}(\overline{\Omega})
  \label{C7}
\end{equation}
holds for any $\sigma<\sigma(s-\eta,p)$.

Finally, when $s=\fracnp$ in the $B$-case an argument similar to
\eqref{C7}, but with $\sigma(s-\eta,p)>s-\varepsilon$, works because
$\lim_{s\to\fracci np_-} \sigma(s,p)=\fracnp=s$. The statement on
$\tilde\sigma$ follows analogously if the effects of (iii) are
disregarded, for in \eqref{C1} ff. any $\varepsilon>0$ and in
\eqref{C7} ff. any $\sigma<\sigma(s,p)$ may be obtained. Similarly
$\sigma=\sigma(s,p)-\varepsilon$ is always possible.

\bigskip

It remains to show \eqref{sgm-cnd}. Here we draw on the
literature, where $\Omega=\Rn$ has been considered by many. 
On $\Rn$ the condition $g(0)=0$ is posed in order to have
$g(0)\in L_p$ also for $p<\infty$, so strictly speaking we should
replace $g(\cdot)$ by $g(\cdot)-g(0)$; this is harmless because $g(0)$
belongs to $\cap_{s,p,q} B^{s}_{p,q}(\overline{\Omega})$. 

Once boundedness has been established on $\Rn$
through an inequality like 
\begin{equation}
  \norm{g(u)}{F^{\sigma(s,p)}_{p,\infty}}\le c
  \norm{u}{F^{s}_{p,\infty}}(1+\norm{u}{F^{s}_{p,\infty}}^{\mu-1})
  \label{C8}
\end{equation}
this carries over to $\Omega$  by restriction: if $r_\Omega v=u$
for $v\in F^{s}_{p,\infty}(\Rn)$, then $g(v)\in
F^{\sigma(s,p)}_{p,\infty}(\Rn)$ restricts to $g(u)$. 
Thus it suffices to consider $\Omega=\Rn$. 

For $s>\fracnp$ it was shown in \cite{Run86} that for every
real-valued $u\in F^{s}_{p,q}(\Rn)$,
\begin{equation}
  \norm{g(u)}{F^{s}_{p,q}}\le
  c\norm{u}{F^{s}_{p,q}}(1+\norm{u}{F^{s}_{p,q}}^{\mu-1}), 
  \label{C9}
\end{equation}
when $\mu>\max(1,s)$, cf.~Theorem~5.4.2 there. Here the general
assumption that $g^{(j)}\in L_\infty(\R)$ for every $j\in\N_0$ is used
to obtain $c$ independent of $u$.

When $(\fracnp-n)_+<s<1$ the estimate in \eqref{C8} is, with
$\sigma(s,p)=s$ and $\mu=1$, a well-known easy consequence of the
characterisation of $F^{s}_{p,q}$ by first order differences,
cf.~\cite[Thm.~3.5.3]{T3} and the estimate
\begin{equation}
  |g(u(x+h))-g(u(x))|\le \norm{g'}{L_\infty}\cdot|u(x+h)-u(x)|.
  \label{C10}
\end{equation}

The cases with $1<s<\fracnp$ are covered by \cite[Lemma~3]{Sic89},
even with a sharper result in Theorem~1 there when
$s>1+(\fracnp-n)_+$. In fact this lemma yields \eqref{C8} for
$\sigma(s,p)=\frac{\fracnp}{\fracnp-s+1}$ and $\mu=\sigma(s,p)$,
provided that $1<s<\fracnp$ and $\sigma(s,p)>(\fracnp-n)_+$ hold. By
definition $\sigma(s,p)>1$ for $s>1$, so this is trivially true for
$1\le p<\infty$; for $p\le1$ the assumption $s<\fracnp$ gives that 
\begin{equation}
  \sigma(s,p)>\fracnp-n\iff s>\frac{(\fracnp)^2-n\fracnp-n}{\fracnp-n}
  \iff s>\fracnp-\tfrac{p}{1-p},
  \label{C11}
\end{equation}
so the second line of (ii) is found from the requirement
$\sigma(s,p)>(\fracnp-n)_+$. 

Finally, for $s=1$ we reduce to the case with $s<1$ by an arbitrarily small
loss of smoothness; for $s=\fracnp$ a reduction to $1<s<\fracnp$ works because
$\lim_{s\to\fracci np_-}\sigma(s,p)=\fracnp=s$. The proof of
Theorem~\ref{dreg-thm} is complete.

\bigskip

We include a few observations on the curve determined by
\eqref{C3} for $\fracnp>0$. For the auxiliary function
$h_1(t)=\frac12(t+3+\sqrt{(t-3)^2-8})$,
\begin{equation}
  \begin{split}
    h_1(t)-t&=\tfrac12(t-3)(\sqrt{1-8(t-3)^{-2}}-1)
    \\
            &=-2(t-3)^{-1}+\cal O((t-3)^{-3})\to 0_-\quad\text{for
         $t\to\infty$}, 
  \end{split}
  \label{C13}
\end{equation}
whereas $h_2(t)=\frac12(t+3-\sqrt{(t-3)^{2}-8})$ satisfies 
\begin{equation}
  h_2(t)-3=\tfrac12(t-3)(1-\sqrt{1-8(t-3)^{-2}})\to 0_+\quad\text{for
  $t\to\infty$}. 
  \label{C14}
\end{equation}
Thus $s=\fracnp$ and $s=3$
are the asymptotes as claimed. The curve itself is a branch of a
hyperbola since the equation in \eqref{C3} may be written
\begin{multline}
    0=(s-3)^2 -(\fracnp-3)(s-3)+2
  \\
  = \begin{pmatrix}\fracnp-3&
  s-3\end{pmatrix} \begin{pmatrix}0&-\tfrac12\\ -\tfrac12 &
  1\end{pmatrix} \begin{pmatrix}\fracnp-3\\ s-3\end{pmatrix}+2,
  \label{C16}
\end{multline}
where the matrix is symmetric and indefinit as the determinant is $-\frac14$.

\subsection{A lemma on continuity}   \label{cont-ssect}
The boundedness obtained for $g(\cdot)$ above means that every bounded
set of $E^s_{p,q}$ is mapped into a bounded set in $E^\sigma_{p,q}$.
Although $g(\cdot)$ is non-linear, this boundedness does imply a norm
continuity if one can afford to loose a little smoothness.

For the reader's convenience we include the next lemma, which is used
in Section~\ref{ex-sect} below; it extends
\cite[3.1]{Sic92} and simplifies \cite[Lem.~5.5.2]{RuSi96}:

\begin{lem}  \label{cont-lem}
When $\Omega$ is as above, and $g\in C^\infty(\R)$ with $g'\in
L_\infty(\R)$, then boundedness, for some $s>(\fracnp-n)_+$,
$0<p\le\infty$ and some $\sigma\in\R$, of
\begin{equation}
  g(\cdot)\colon E^s_{p,q}\to E^\sigma_{p,q}
  \label{C31}
\end{equation}
implies norm continuity of
\begin{equation}
  g(\cdot)\colon E^s_{p,q}\to E^{\sigma-\varepsilon}_{p,q}
  \quad\text{for each $\varepsilon>0$}.
  \label{C32}
\end{equation}
\end{lem}
\begin{proof}
In the Besov case one has, when $t<\min(0,\sigma-\varepsilon)$, that 
\begin{equation}
  B^{\sigma-\varepsilon}_{p,q}(\overline{\Omega})=
  (B^{\sigma}_{p,q}(\overline{\Omega}),
   B^{t}_{p,q}(\overline{\Omega}))_{\theta,q} 
  \label{C33}
\end{equation}
for some $\theta\in\,]0,1[\,$, cf. \cite[Thm.~3.3.6]{T2}. When $r=\max(1,r)$
\begin{equation}
  \norm{g(u)-g(v)}{B^{\sigma-\varepsilon}_{p,q}}\le c
  \norm{g(u)-g(v)}{L_r}^{1-\theta}\norm{g(u)-g(v)}{B^{\sigma}_{p,q}}^{\theta},
  \label{C34}
\end{equation}
since $L_r(\Omega)\hookrightarrow B^{t}_{p,q}(\overline{\Omega})$
then. In $L_r$ an estimate like \eqref{C10} is applicable, and
thereafter $B^{s}_{p,q}\hookrightarrow L_r$ may be used (for $p<1$
this embedding is based on the assumption $s>\fracnp-n$). Thus the
first factor on the right hand side tends to $0$ for $u\to v$ in
$B^{s}_{p,q}$ while the second remains bounded by \eqref{C31}.

In the $F^{s}_{p,q}$ case, $g(\cdot)\colon F^{s}_{p,q}\to
B^{\sigma}_{p,\infty}$ is bounded, so analogously
\begin{equation}
  g(\cdot)\colon F^{s}_{p,q}(\overline{\Omega})\to
    B^{\sigma-\eta}_{p,q}(\overline{\Omega}) 
  \label{C35}
\end{equation}
is continuous for any $\eta>0$. Then \eqref{C32} follows.
\end{proof}

\subsection{Interrelations between conditions (i), (ii) and (iii)}   
\label{inter-ssect}

\begin{rem}  \label{r=2-rem}
In the definition of $\Dm(A_T+g(\cdot))$ the condition:
\begin{equation}
  s>\fracnp-\tfrac{p}{1-p}\quad\text{for}\quad 0<p<1
  \label{C41}
\end{equation}
in (ii) of Theorem~\ref{dreg-thm} is always redundant when $T$ has class
$r=2$.

Indeed, since one has
\begin{equation}
  \fracnp-\tfrac{p}{1-p}\le \fracnp-n+2\iff p(n-1)\ge n-2
  \label{C42}
\end{equation}
it is clear that when $(s,p,q)$ satisfies (i) for $r=2$, then \eqref{C41}
holds if either $n=1$, $n=2$ or if $\frac{n-2}{n-1}\le p<1$ when $n\ge 3$.

Therefore, when (i) and (iii) hold for $r=2$, then it suffices to verify for
$n\ge3$ and $0<p<\tfrac{n-2}{n-1}$ that the first inequality in (iii) poses a 
stronger condition than \eqref{C41}.
This follows from Remark~\ref{condition-rem}.
\end{rem}

\begin{rem}   \label{n=1-rem}
For $n=1$ condition (i) in Theorem~\ref{dreg-thm} amounts to 
\begin{equation}
  s>\fracnp,
  \label{C51}
\end{equation}
since $r\ge1$. Therefore any $E^s_{p,q}$ in $\Dm(A_T+g(\cdot))$ satisfies
$E^s_{p,q}\hookrightarrow C(\overline{\Omega})$, and both (ii) and (iii) hold
when (i) does so. 

Hence Figure~\ref{DM-fig} is misleading for $n=1$, and in fact
\begin{equation}
  \Dm(A_T+g(\cdot))=\bigl\{\,(s,p,q)\bigm| s>\fracp-1+r\,\bigr\},
  \label{C52}
\end{equation}
which in contrast to the general case (for $n\ge2$) is convex.
\end{rem}

\begin{rem}   \label{n=2-rem}
Also $n=2$ gives an exception from the overview after Theorem~\ref{dreg-thm}.

In this case $\Dm(A_T+g(\cdot))$ is still not convex for $r=1$, but (ii)
implies (iii), so that the curved boundary is given by
$s=\fracnp-\tfrac{p}{1-p}$. See Remark~\ref{condition-rem} below for the
details. 

Moreover, for $n=2=r$ it follows from Remark~\ref{r=2-rem} that even (ii) is
redundant, cf.~\eqref{C42}, and hence 
\begin{equation}
  \Dm(A_T+g(\cdot))=\bigl\{\,(s,p,q)\bigm| s>\max(\fracp+1,\fracc 2p)\,\bigr\}.
  \label{C61}
\end{equation}
Evidently this is convex, so also this case deviates from the general pattern. 
\end{rem}

\begin{rem}   \label{condition-rem}
Among the requirements in Theorem~\ref{dreg-thm}, the condition
\begin{alignat*}{2}
  \upn{(iii)}'&\qquad & s&>\tfrac{1}{2}(\fracnp+3+\sqrt{(\fracnp-3)^2-8}\,)  \\
\intertext{turns out to be almost always stronger than} 
  \upn{(ii)}' &\qquad & s&>\fracnp-\frac{n}{\fracnp-n}
\end{alignat*}
when they both apply, that is for $\fracnp\in
\,]\max(n,3+\sqrt{8}),\infty[$ and $n\ge2$. The exceptions are for $n=3$ in
which case $\text{(ii)}'\implies \text{(iii)}'$ in the narrow interval with
$3+\sqrt{8}\le\fracnp< 6$ and in general for $n=2$.

Observe first that $\text{(ii)}'$ and $\text{(iii)}'$ are redundant
for $n=1$ by Remark~\ref{n=1-rem}.
To analyse when $\text{(iii)}'\implies\text{(ii)}'$ for $n\ge2$, consider
\begin{equation}
  t-3-\tfrac{2n}{t-n}\le\sqrt{(t-3)^2-8}
  \label{C22}
\end{equation}
when $t>n$ and $t\ge 3+\sqrt{8}$ as well as $n=2,3,\dots$. Notice that the
left hand side equals $(t-n)^{-1}(t^2-(n+3)t+n)$ and is negative when 
\begin{equation}
  t^2-(n+3)t+n<0;
  \label{C23}
\end{equation}
the discriminant $n^2+2n+9$ is $>0$. Thus \eqref{C22}
always holds for $t\in[\alpha_-(n),\alpha_+(n)]$ when
$2\alpha_\pm(n)=n+3\pm\sqrt{n^2+2n+9}$. Here $\alpha_+(n)>n$ and
$\alpha_-(n)<\min(n,3+\sqrt{8})$.

For $t\ge\max(\alpha_+(n),3+\sqrt8)$ it is found by taking squares that
\begin{equation}
  \begin{gathered}
  \text{\eqref{C22}}\iff
  \tfrac{4n^2}{(t-n)^2}-2(t-3)\tfrac{2n}{t-n}\le-8 
  \\
  \iff 0\le(n-2)t^2-n(n-1)t.
  \end{gathered}
  \label{C25}
\end{equation}
The last inequality is false for $n=2$, and since $\alpha_+(2)<3+\sqrt{8}$
it is proved that $\text{(ii)}'\implies\text{(iii)}'$ for $n=2$.

Since $t=0$ and $t=n(n-1)/(n-2)$ are the roots of the polynomial
$(n-2)t^2-n(n-1)t$, the implication $\text{(iii)}'\implies\text{(ii)}'$ holds
for all $t\le\max(\alpha_+(n),3+\sqrt8)$ precisely when 
\begin{equation}
  \tfrac{n(n-1)}{n-2}\ge\max(\alpha_+(n),3+\sqrt8)
  \label{C26}
\end{equation}
does so. A straightforward calculation shows that 
\begin{equation}
  \tfrac{n(n-1)}{n-2}<\alpha_+(n) \iff n\ge 4,
  \label{C27}
\end{equation}
so \eqref{C26} holds for all $n\ge4$. In addition $\alpha_+(3)=3+\sqrt{6}$
while $\tfrac{n(n-1)}{n-2}\big|_{n=3}=6$, so by \eqref{C25} the inequality
\eqref{C22} holds for $t\in [6,\infty[\,$ when $n=3$.

Altogether this shows that, except for $n=2$ and a small interval for $n=3$,
the condition $s>\fracnp-\frac{p}{1-p}$, that is $\cal O(\fracnp)$, only
interferes with the second requirement in (iii). In other 
words, when $n\ge 3$ the domains $\Dm(A_T+g(\cdot))$ are for
$\fracnp\ge6$ only 
defined by the stronger condition $\text{(iii)}'$.   
\end{rem}

\section{Proof of the Inverse Regularity Theorem} \label{ireg-sect}
Before the regularity properties of Theorem~\ref{ireg-thm} are proved in
Section~\ref{irpf-ssect} below, we
review the prerequisites on elliptic problems in Besov and
Triebel--Lizorkin spaces for a better reading. 

\subsection{The Boutet de Monvel calculus}  \label{bdm-ssect}
There are two sour\-ces for elliptic theory in the full
$B^{s}_{p,q}$ and $F^{s}_{p,q}$ scales; the Agmon--Douglis--Nirenberg
theory has been extended in \cite{FR95}, but this is not quite
adequate here, cf.\ Remark~\ref{parametrix-rem}. Instead we use the
pseudo-differential boundary operator calculus, which was generalised
to these spaces in \cite{JJ96ell} and \cite[Ch.~4]{JJ93}. 

As a general introduction to the calculus there is \cite{G2} and 
the introduction and Section~1.1 in \cite{G1}.

\subsubsection{Green Operators}   \label{grn-sssect}
In a systematic approach to boundary problems, the basic ingredient to
study is a matrix operator
\begin{equation}
  \cal A= \begin{pmatrix} P_\Omega+G & K\\ T& S\end{pmatrix}\colon
  \begin{array}{c}
  C^\infty(\overline{\Omega})^N \\ \oplus \\ C^\infty(\Gamma)^M
  \end{array}
  \to
  \begin{array}{c}
  C^\infty(\overline{\Omega})^{N'} \\ \oplus \\ C^\infty(\Gamma)^{M'}
  \end{array}
 \label{ir1}
\end{equation}
where $P_\Omega:=r_\Omega P e_\Omega$ is the \emph{truncation} to
$\Omega$ of a pseudo-differential operator on $\Rn$, $K$ is a {\em
Poisson\/} operator, $T$ is a \emph{trace} operator, $S$ is a
pseudo-differential operator in $\Gamma$ whilst $G$ is a \emph{singular
Green} operator.

As examples of \eqref{ir1}, or of the so-called Green operators, one can take 
\begin{equation}
   \begin{pmatrix}\mlap\\ \gamma_0\end{pmatrix},\quad
    \begin{pmatrix}\mlap\\ \gamma_1\end{pmatrix}\quad\text{or}\quad
    \begin{pmatrix}A \\ T\end{pmatrix},
 \label{ir2}
 \end{equation}
whereby $M=0$ since they are column matrices, or their parametrices
\begin{equation}
  \begin{pmatrix}R_D&K_D\end{pmatrix},\quad
 \begin{pmatrix}R_N&K_N\end{pmatrix}\quad\text{resp.}\quad
 \begin{pmatrix}R&K\end{pmatrix}
 \label{ir3}
 \end{equation}
(when $\left(\begin{smallmatrix}A\\T\end{smallmatrix}\right)$ is
elliptic); hereby $M'=0$ because of the row-form.

For realisations like $A_T$ considered above a variety of results
follow easily from a study of
$\left(\begin{smallmatrix}A\\T\end{smallmatrix}\right)$, so we focus
on the latter operator to begin with.

To get  a good calculus of Green operators like $\cal A$ above,
Boutet de~Monvel \cite{BM71} introduced first of all the
requirement that $P$ should have the
transmission property at $\Gamma\subset\Rn$. That is to say, for
$N=N'=1$, $P_\Omega$ should map $C^\infty(\overline{\Omega})$ into
itself\,---\,when $P$ merely
belongs to the H{\"o}rmander class $S^{d}_{1,0}(\Rn\times\Rn)$,
then $P_\Omega(C^\infty(\overline{\Omega}))\subset
H^{-d}(\overline{\Omega})\cap C^\infty(\Omega)$ (since the singular
support of $P(e_\Omega\varphi)$, for $\varphi\in
C^\infty(\overline{\Omega})$, as 
a subset of $\Gamma$, is not felt after application of $r_\Omega$); 
thus the transmission property rules out blow-up at $\Gamma$.

Secondly, the notion of singular Green operators $G$ was introduced in
order to encompass solution operators; e.g., when the inverse of
$\left(\begin{smallmatrix}\mlap\\
\gamma_0\end{smallmatrix}\right)$ is 
denoted $\left(\begin{smallmatrix}R_D &K_D\end{smallmatrix}\right)$, 
then $R_D$ is \emph{not} a truncated pseudo-differential operator. In fact,
$R_D=\op{OP}(|\xi|^{-2})_\Omega+G_D$, where the compensating term
$G_D$ is a singular Green operator equal to
$-K_D\gamma_0\op{OP}(|\xi|^{-2})_\Omega$.

For the precise symbol classes of $P_\Omega$, $G$, $K$, $T$ and $S$,
with the uniformly estimated class $S^{d}_{1,0}(\Rn\times\Rn)$ as the basis,
the reader is referred to \cite{GK}. A discussion of the
transmission property is found in a work of Grubb and H{\"o}rmander
\cite{GH}; let us also mention \cite{G2}, \cite[Sect.~3.2]{JJ96ell} 
and Section~1.2 in the second edition of \cite{G1}.

We proceed to state relevant properties of $\cal A$. 
Further details and proofs are given in \cite{JJ96ell}. Specialising to 
$\cal A=\left(\begin{smallmatrix}A\\ T\end{smallmatrix}\right)$ with $A$
and $T$ as in Section~\ref{intr-sect}, $P_\Omega=A$ is of order~$2$,
$G=0$ and ($K$ and $S$ being redundant, i.e. $M=0$) $T$ is of
order $d$ and class $r=1$ or $2$. Then
\begin{align}
  \cal A&\colon B^{s}_{p,q}(\overline{\Omega})\to 
  B^{s-2}_{p,q}(\overline{\Omega})\oplus B^{s-d-\fracpi}_{p,q}(\Gamma)
  \label{ir4}  \\
  \cal A&\colon F^{s}_{p,q}(\overline{\Omega})\to 
  F^{s-2}_{p,q}(\overline{\Omega})\oplus B^{s-d-\fracpi}_{p,p}(\Gamma)
  \label{ir5}
\end{align} 
are bounded when $s>r+\max(\fracp-1,\fracnp-n)$.

The assumed ellipticity of $\cal A$ in the sense of the calculus amounts to 
\begin{itemize}
    \item[(I)] $A$'s principal symbol,
      $a^0(x,\xi)=\sum_{|\alpha|=2} a_\alpha(x)\xi^\alpha$, is non-zero,
      \begin{equation}
        a^0(x,\xi)\ne 0, \quad\text{for $x\in\Omega$ and $|\xi|\ge1$};
        \label{ellip1}
      \end{equation}
  \item[(II)] the principal boundary symbol operator
    $a^0(D_n)=a^0(x',0,\xi',D_n)$,
    \begin{equation}
      a^0(D_n)\colon \cal S(\Rp)\to
      \begin{array}{c}
        \cal S(\Rp)\\ \oplus \\ \C
      \end{array}
      \label{ellip2}
    \end{equation}
    is a \emph{bijection} for each $x\in\Omega$ and $|\xi'|\ge1$.
\end{itemize}
Here $a^0(D_n)$ is defined from the principal part of
$\left(\begin{smallmatrix}A\\ T\end{smallmatrix}\right)$ by 
means of local coordinates in which $\Gamma$ is a subset of
$\{x_n=0\}$; there $x_n$ is set equal to $0$ and $D_j$ is replaced by
$\xi_j$ when $j<n$. 

The ellipticity assures the existence of a parametrix $\widetilde{\cal
A}$, that is, another Green operator in the calculus such that 
\begin{equation}
  \widetilde{\cal A}\cal A=1-\cal R,\qquad
  \cal A\widetilde{\cal A}=1-\cal R'
  \label{ir6}
\end{equation}
for negligible operators $\cal R$ and $\cal R'$; i.e.~Green operators
of order $-\infty$. Although $\cal A$ is purely differential,
$\widetilde{\cal A}$ has the form $\begin{pmatrix} R&K
\end{pmatrix}$ where $R=P_\Omega+G$ for a truly
pseudo-differential operator $P$ with transmission property at
$\Gamma$ and a non-trivial singular Green operator. The orders of 
$R$ and $K$ are $-2$ and $-d$, respectively, while $R$ may be taken of
class~$r-2$ (best possible), cf.~\cite[Thm.~5.4]{G3}. Hence, by
\eqref{ir4}--\eqref{ir5}, 
\begin{align}
  \widetilde{\cal A} &\colon B^{s-2}_{p,q}(\overline{\Omega})\oplus
    B^{s-d-\fracpi}_{p,q}(\Gamma)\to B^{s}_{p,q}(\overline{\Omega}) 
  \label{ir7} \\
  \widetilde{\cal A} &\colon F^{s-2}_{p,q}(\overline{\Omega})\oplus
    B^{s-d-\fracpi}_{p,p}(\Gamma)\to F^{s}_{p,q}(\overline{\Omega}) 
  \label{ir8}
\end{align} 
are bounded for $s>r+\max(\fracp-1,\fracnp-n)$.

Using $\widetilde{\cal A}$ it may be shown that there exist two
finite-dimensional subspaces
\begin{equation}
  \ker \cal A  \subset C^\infty(\overline{\Omega})
\qquad
  \cal N \subset C^\infty(\overline{\Omega})\oplus C^\infty(\Gamma),
  \label{ir10}
\end{equation}
(and that $\cal A(E^s_{p,q})$ is closed) such that whenever 
$s>r+\max(\fracp-1,\fracnp-n)$,
\begin{gather}
   \ker\cal A=\bigl\{\,u\in E^s_{p,q}\bigm| \cal Au=0\,\bigr\},
  \label{ir11} \\
  \begin{aligned}
    \cal A(B^{s}_{p,q})\oplus\cal N&=
    B^{s-2}_{p,q}(\overline{\Omega})\oplus B^{s-d-\fracpi}_{p,q}(\Gamma),
    \\
    \cal A(F^{s}_{p,q})\oplus\cal N&=
    F^{s-2}_{p,q}(\overline{\Omega})\oplus B^{s-d-\fracpi}_{p,p}(\Gamma).
  \end{aligned}
  \label{ir12}
\end{gather}
In other words, the kernel of $\cal A$ is $(s,p,q)$-independent and the
range complement may be picked with this property.

\subsubsection{Realisations}   \label{realisations-sssect}
For $A_T$ in \eqref{AT-def}--\eqref{AT-dom} the subspaces $B^{s}_{p,q;T}$ and
$F^{s}_{p,q;T}$ defined by $Tu=0$ make sense for
$s>r+\max(\fracp-1,\fracnp-n)$, and  
\begin{equation}
  A_T\colon E^s_{p,q;T}\to E^{s-2}_{p,q}
  \label{ir20}
\end{equation}
is bounded for such $(s,p,q)$, by \eqref{ir4} and \eqref{ir5}.

Ellipticity of $A_T$ means that
$\left(\begin{smallmatrix}A\\T\end{smallmatrix}\right)$  is elliptic,
i.e.\ that (I) and (II) are satisfied. In the elliptic case even $A_T$
has a parametrix, say $R_0$; it is of the form $(A_0)_\Omega+G_0$,
where $A_0$ is a parametrix of $A$ on $\Rn$ and $G_0$ is a singular
Green operator, both of order $-2$ and $(A_0)_\Omega+G_0$ of class
$r-2$, so 
\begin{equation}
  R_0\colon E^{s-2}_{p,q}\to E^s_{p,q}
  \label{ir21}
\end{equation}
is bounded whenever $s>r+\max(\fracp-1,\fracnp-n)$ by the general
result in \eqref{ir4}--\eqref{ir5}. More importantly, $R_0$ can be
taken so that 
\begin{itemize}
  \item $R_0$ maps $E^{s-2}_{p,q}$ into $D(A_T)=E^s_{p,q;T}$;
  \item both $R_0A_T-I$ and $A_TR_0-I$ have finite-dimensional ranges
    in $C^\infty(\overline{\Omega})$. 
\end{itemize}
This follows as in \cite[Prop.~1.4.2]{G1}; when $r\ne2$ or $d\ne2$ one
can modify the order and class reduction in (1.4.14) there,
as in \cite[(5.32)]{G3}.

For the Fredholm properties of $A_T$ one has obviously that $\ker
A_T=\ker\cal A$, but it is a point to show that $A_T(E^s_{p,q;T})$ is
complemented also for $p$, $q<1$ in which case $E^s_{p,q}$ is not locally
convex. 
However, when $T$ has a Poisson operator $K$ as a right inverse,
i.e.\ $TK=I$, then 
\begin{equation}
  \Phi= \begin{pmatrix} I & -AK\end{pmatrix},
  \label{ir22}
\end{equation} 
may be used in a way similar to the proof of \cite[4.3.1]{G1} to get

\begin{lem}  \label{cmpl-lem}
$1^\circ$ When $(s,p,q)$ is admissible and $W$ is a range complement of
$\cal A$, then $A_T(E^s_{p,q;T})$ is closed while $\dim \Phi(W)=\dim W$ and
$E^{s-2}_{p,q}=A_T(E^s_{p,q;T})\oplus \Phi(W)$.

$2^\circ$ A subspace $\cal N\subset C^\infty$ is a range complement of 
$A_T$ for some $(s,p,q)$ if and only if it
is so for every $(s,p,q)$ admissible for $A_T$.
\end{lem}
\begin{proof} As in \cite[4.3.1]{G1}, $\Phi$ is seen to be injective on
$W$, hence $\dim \Phi(W)=\dim W$, and $\Phi(W)$ to be linearly independent
of $R(A_T):=A_T(E^s_{p,q;T})$. Then, using the quotient $Q$ onto
$E^{s-2}_{p,q}/R(A_T)$, $\dim \Phi(W)\le \dim Q(E^{s-2}_{p,q})$
follows. But a finite dimensional $U\subset Q(E^{s-2}_{p,q})$ equals $QV$
for some $V$ linearly independent of $R(A_T)$ and with $\dim U=\dim
V\le\dim W$ (since $V\times\{0\}$ is linearly independent of $\cal
A(E^{s}_{p,q})$).  Altogether $\dim Q(E^{s-2}_{p,q})=\dim \Phi(W)<\infty$,
so $R(A_T)$ is closed by \cite[19.1.1]{H} (carried over to $E^s_{p,q}$ by
\cite[1.41(d)+2.12(b)]{R}) and complemented by $\Phi(W)$.

Since $\cal N=\Phi(\cal N\times\{0\})$, $W=\cal N\times\{0\}$ is possible
for dimensional reasons. By Theorem~1.3 or 5.2 of \cite{JJ96ell}, $W$ is a
range complement for every $(s,p,q)$; by $1^\circ$, so is $\cal N$. 
\end{proof}

Existence of such a $K$ is assured when $T$ is \emph{normal}; see 
Proposition~1.6.5, 
Definition~1.4.3 and Remark~1.4.4 in \cite{G1}. For $d=0$ normality
means that $T=S_0\gamma_0$, where $S_0(x)$ is a \emph{function}
without roots on $\Gamma$; when $d=1$, $T$ is normal when
$S_1(x)$ is such a zero-free function. 

Finally, one can in this case project onto the kernel
and range of $A_T$.

\begin{prop}  \label{cmpl-prop}
Let $A_T$ be an elliptic realisation of $A$ as described above, with a
right inverse of $T$ (or $T$ normal). 

For each $C^\infty$ range complement $\cal N$ and each
$s>r+\max(\frac1p-1,\fracnp-n)$ there is a continuous idempotent 
\begin{equation}
  Q\colon E^{s-2}_{p,q}\to E^{s-2}_{p,q}, \quad\text{projecting onto
  $\cal N$ along $A_T(E^s_{p,q;T})$.}
  \label{ir23}
\end{equation}
When $\{w_1,\dots,w_m\}$ is an $L_2$-orthonormal basis for $\ker A_T$,
\begin{equation}
  Pu=\sum_{j=1}^m \dual{u}{w_j} w_j\quad\text{is bounded}\quad
  P\colon E^s_{p,q}\to E^s_{p,q}
  \label{ir25}  
\end{equation}
and projects onto $\ker A_T$ whenever $s>r+\max(\frac1p-1,\fracnp-n)$.

Furthermore, when $A_T$ is self-adjoint in $L_2(\Omega)$, one
can take $\cal N=\ker A_T$ for every $(s,p,q)$ as above and then
\eqref{ir25} holds even on $E^{s-2}_{p,q}$.  
\end{prop}

\begin{proof} When \eqref{M12} holds \cite[Thm.~5.16]{R} gives the
existence and continuity of $Q$. 
This does not just carry over to $\ker A_T$, for application of, say,
\cite[Lem.~4.21]{R} requires local convexity.

However, the given $P$ is defined for $u\in E^s_{p,q}$ when
$s>r+\max(\frac1p-1,\fracnp-n)$, for since $r\in\{\,1,2\,\}$ we have
$s>0$ so that $E^s_{p,q}\hookrightarrow L_1(\Omega)$,
$\dual{u}{w_j}=\int_\Omega uw_j$ is defined and 
\begin{equation}
  |\dual{u}{w_j}|\le \norm{u}{L_1}\norm{w_j}{L_\infty}\le
  c\norm{u}{E^s_{p,q}}\norm{w_j}{L_\infty};
  \label{ir27}
\end{equation}
continuity of $P$ follows. By construction $P^2=P$ and 
$\ker A_T=P(E^s_{p,q})$.

When $A_T=A_T^*$ in $L_2$, then $\ker A_T$ is a range complement in
$E^{s-2}_{p,q}$ by the lemma. Consider first $r=2$. Then the inequality
for $s$ implies that $E^{s-2}_{p,q}$ is contained in the dual 
of some $E^{s_2}_{p_2,q_2}\supset\ker A_T$, and analogously to the above $P$
is a continuous projection in $E^{s-2}_{p,q}$ onto $\ker A_T$.

For $r=1$ elements of e.g.\ $H^{-1}$ may occur in \eqref{ir25}. However, 
$w\in \ker A_T$ implies $\gamma_0 w=0$: evidently $Tw=0$ where
$T=S_0\gamma_0$ and $S_0(x)$ is a function on $\Gamma$ (being a differential
operator of order $0$ by assumption), and $S_0(x)$ cannot have any zeroes
because $S_0\gamma_0$ has a right inverse. Thus $\gamma_0w=0$.

So when $s>1+\max(\frac1p-1,\fracnp-n)$, the space $E^{s-2}_{p,q}$ is
embedded into some $E^{s_1-2}_{p_1,q_1}$ with $s_1>1+\max(\fracc1{p_1}-1,
\fracc n{p_1}-n)$ and $p_1$, $q_1\in\,]1,\infty]$. The latter is
dual to $E^{s_2}_{p_2,q_2;0}=\{\,u\in E^{s_2}_{p_2,q_2}\mid 
\gamma_0u=0\,\}$ when $s_1-2+s_2=0$, $\fracc1{p_1}+\fracc1{p_2}=1$ and
$\fracc1{q_1}+\fracc1{q_2}=1$, and since $\ker A_T\subset E^{s_2}_{p_2,q_2;0}$,
$P$ in \eqref{ir25} is defined on $E^{s_1-2}_{p_1,q_1}$, hence on
$E^{s-2}_{p,q}$. Again $P$ is bounded and idempotent. 
\end{proof}
 
\subsection{Proof of Theorem~\ref{ireg-thm}}   \label{irpf-ssect}
We now turn to one of the main subjects in this article: the Inverse
Regularity Theorem for the problem in \eqref{cmp-pb}. For the proof
the bootstrap method in \cite{JJ93,JJ95stjm,JJ95reg} is extended to overcome
the difficulties caused by the non-convexity of $\Dm(A_T+g(\cdot))$. 

\bigskip

Basically the non-linear estimates and the elliptic theory is used as
follows: suppose $u(x)$ in $E^{s_1}_{p_1,q_1;T}$ is a solution of 
\begin{equation}
  A_T u+ g(u) = f
  \label{ir51}
\end{equation}  
for $f(x)$ in $E^{s_0-2}_{p_0,q_0}$ with both $(s_0,p_0,q_0)$ and
$(s_1,p_1,q_1)$ in $\Dm(A_T+g(\cdot))$. Then $R_0$, the
parametrix of $A_T$ introduced in \eqref{ir21} ff., is bounded 
\begin{equation}
  R_0\colon E^{s_0-2}_{p_0,q_0}\to E^{s_0}_{p_0,q_0;T}
  \label{ir52}
\end{equation}
because $(s_0,p_0,q_0)\in\Dm(A_T+g(\cdot))$. Thus $R_0$ can be applied
to the right hand side of \eqref{ir51}, hence to the left hand
side. By Theorem~\ref{dreg-thm} and \eqref{ir20}, both $A_Tu$ and
$g(u)$ are in $E^{s_1-2}_{p_1,q_1}$, and so $R_0$ acts linearly
on the left hand side of \eqref{ir51}. After a rearrangement, cf.\
Remark~\ref{parametrix-rem} below, we get 
\begin{equation}
  u= R_0f-R_0g(u)+\cal Ru
  \label{ir53}  
\end{equation}
where $\cal R:=R_0A_T-I$ is an operator with range in
$C^\infty(\overline{\Omega})$. 

Since $R_0g(u)\in E^{\sigma_1+2}_{p_1,q_1}$ for some
$\sigma_1>s_1-2$ by Theorem~\ref{dreg-thm}, one may now search for 
$E^{s_2}_{p_2,q_2;T}$ large enough to contain 
$E^{s_0}_{p_0,q_0;T}+ E^{\sigma_1+2}_{p_1,q_1;T}$, and thus
\begin{equation}
  R_0f-R_0g(u)+\cal Ru \in E^{s_2}_{p_2,q_2;T}.
  \label{ir54}
\end{equation}
Then $u\in E^{s_2}_{p_2,q_2;T}$, and this fact is used to get a
new knowledge about $R_0g(u)$ and then for $u$
itself. Thus we seek spaces $E^{s_1}_{p_1,q_1;T}$,
$E^{s_2}_{p_2,q_2;T}$, \dots containing $u(x)$, and the task is to
obtain $E^{s_j}_{p_j,q_j}\hookrightarrow E^{s_0}_{p_0,q_0}$ for some $j$.

Obviously it is irrelevant for the application of $g(\cdot)$ whether we
consider $u$ in the subspace $E^{s_j}_{p_j,q_j;T}$ or not, so for
simplicity we use the full space $E^{s_j}_{p_j,q_j}$.

Furthermore we shall first treat the case where
$E^{s_0-2}_{p_0,q_0}=F^{s_0-2}_{p_0,\infty}(\overline{\Omega})$ and
$E^{s_1}_{p_1,q_1;T}=F^{s_1}_{p_1,\infty;T}$; the other cases follow
from this at the end. This allows us to work with the function
$\sigma(s,p)$ from \eqref{sgm-cnd}, or more relevantly
\begin{equation}
 \delta(s,p):=\sigma(s,p)-(s-2),
  \label{delta-cnd} 
\end{equation}
which measures the deviation of $g(\cdot)$'s order from that of $A_T$.
Thus $\sigma_1+2$ above \eqref{ir54}  should be replaced by
$s_1+\delta(s_1,p_1)$, but for convenience we let
$\delta_j=\delta(s_j,p_j)$ in the following. 

\subsubsection{The Worst Case}   \label{wrst-sssect}
The sets corresponding to $\Dm(A_T+g(\cdot))$ in
\cite{JJ93,JJ95stjm,JJ95reg} are all convex, so to begin with we first
consider the case when 
\begin{equation}
  \text{$(\fracc n{p_0},s_0)$ and $(\fracc n{p_1},s_1+\delta_1)$}
  \label{ir61}
\end{equation}
cannot be connected by a straight line in the $(\fracnp,s)$-plane.
The worst case is when this is caused by the hyperbola defined by
condition~(iii) in Theorem~\ref{dreg-thm}. (The other possibility
stems from the condition $s>\fracnp-\tfrac{p}{1-p}$.)

If $s_1+\delta_1>s_0$ we note that also
$s_1+\delta_1-\fracc n{p_1}>s_0-\fracc n{p_0}$ (otherwise there would be a
connecting straight line), and therefore $E^{s_1+\delta_1}_{p_1,q_1}$ is
embedded into $E^{s_0}_{p_0,q_0}$. Thus $(s_2,p_2,q_2)=(s_0,p_0,q_0)$
is possible and the conclusion (reached above) that $u\in
E^{s_2}_{p_2,q_2}$ is already the desired one.

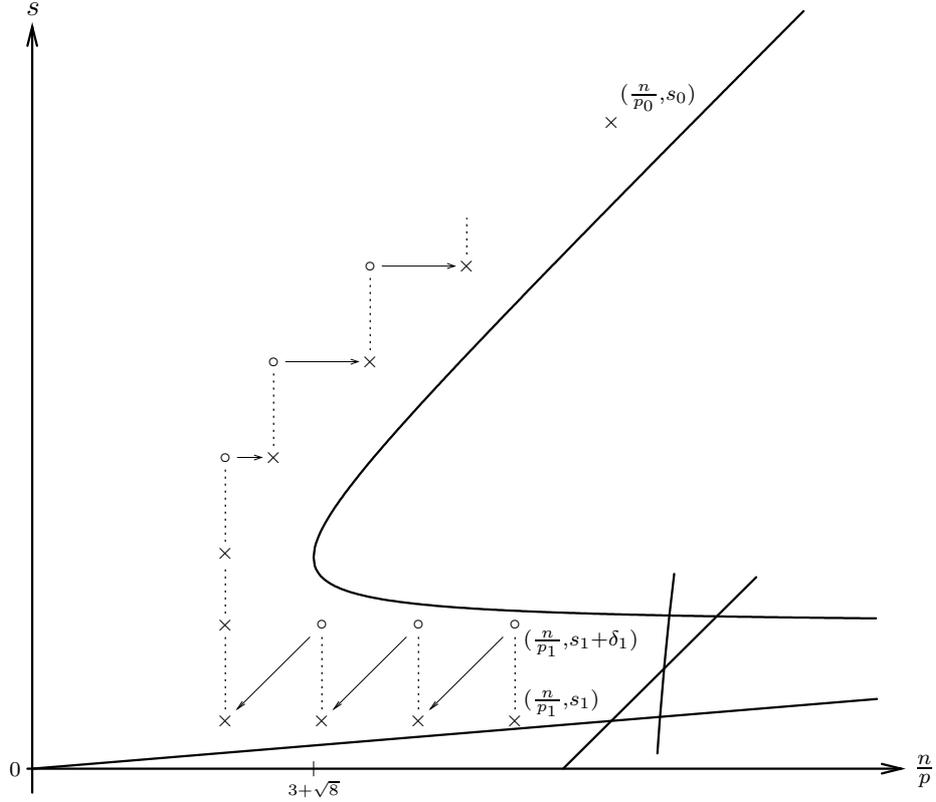
\begin{figure}[htbp]
\hfil
\setlength{\unitlength}{0.0125in}
\begin{picture}(380,330)(0,0)

\path(126.569,7)(126.569,13)                         


\put(210,30){\makebox(0,0)[cc]{$\scriptstyle\times$}}
\put(213,32){\makebox(0,0)[lb]{$\scriptstyle (\fracci n{p_1},s_1)$}}

\put(250,280){\makebox(0,0)[cc]{$\scriptstyle\times$}}
\put(253,285){\makebox(0,0)[lb]{$\scriptstyle (\fracci n{p_0},s_0)$}}

\dottedline{3}(210,35)(210,65)
\put(210,70){\makebox(0,0)[cc]{$\scriptstyle\circ$}}
\put(213,68){\makebox(0,0)[lt]{$\scriptstyle (\fracci n{p_1},s_1+\delta_1)$}}
\path(205,65)(175,35)\path(177,38)(175,35)(178,37)
\put(170,30){\makebox(0,0)[cc]{$\scriptstyle\times$}}
\dottedline{3}(170,35)(170,65)
\put(170,70){\makebox(0,0)[cc]{$\scriptstyle\circ$}}
\path(165,65)(135,35)\path(137,38)(135,35)(138,37)
\put(130,30){\makebox(0,0)[cc]{$\scriptstyle\times$}}
\dottedline{3}(130,35)(130,65)
\put(130,70){\makebox(0,0)[cc]{$\scriptstyle\circ$}}
\path(125,65)(95,35)\path(97,38)(95,35)(98,37)
\put(90,30){\makebox(0,0)[cc]{$\scriptstyle\times$}}
\dottedline{3}(90,35)(90,65)
\put(90,70){\makebox(0,0)[cc]{$\scriptstyle\times$}}
\dottedline{3}(90,75)(90,95)
\put(90,100){\makebox(0,0)[cc]{$\scriptstyle\times$}}
\dottedline{3}(90,105)(90,135)
\put(90,140){\makebox(0,0)[cc]{$\scriptstyle\circ$}}
\path(95,140)(105,140)\path(102,139)(105,140)(102,141)
\put(110,140){\makebox(0,0)[cc]{$\scriptstyle\times$}}
\dottedline{3}(110,145)(110,175)
\put(110,180){\makebox(0,0)[cc]{$\scriptstyle\circ$}}
\path(115,180)(145,180)\path(142,179)(145,180)(142,181)
\put(150,180){\makebox(0,0)[cc]{$\scriptstyle\times$}}
\dottedline{3}(150,185)(150,215)
\put(150,220){\makebox(0,0)[cc]{$\scriptstyle\circ$}}
\path(155,220)(185,220)\path(182,219)(185,220)(182,221)
\put(190,220){\makebox(0,0)[cc]{$\scriptstyle\times$}}
\dottedline{3}(190,225)(190,240)

\thicklines
\path(10,10)(360,39.1667)                            
\path(230,10)(310,90)                                

\path(10,0)(10,320)\path(12,312)(10,320)(8,312)
\path(8,10)(370,10)\path(362,8)(370,10)(362,12)
\put(7,325){\makebox(0,0)[lb]{$s$}}
\put(375,9){\makebox(0,0)[lc]{$\fracnp$}}

\put(0,7){\makebox(0,0)[lb]{$\scriptstyle 0$}}
\put(126.568,5){\makebox(0,0)[ct]{$\scriptscriptstyle 3+\sqrt8$}}


\path(126.569,98.284)(127.069,102.303)(127.569,104.126)
(128.069,105.591)(128.569,106.872)(129.069,108.036)(129.569,109.117)
(130.569,111.107)(131.569,112.936)(132.569,114.652)(133.569,116.284)
(134.569,117.849)(135.569,119.362)(136.569,120.830)(137.569,122.261)
(138.569,123.660)(139.569,125.031)(140.569,126.379)(141.569,127.705)
(142.569,129.012)(143.569,130.302)(144.569,131.577)(145.569,132.837)
(146.569,134.085)(147.569,135.321)(148.569,136.547)(149.569,137.763)
(150.569,138.969)(151.569,140.167)(152.569,141.357)(153.569,142.540)
(154.569,143.716)(155.569,144.886)(156.569,146.049)(157.569,147.207)
(158.569,148.359)(159.569,149.506)(160.569,150.649)(161.569,151.787)
(162.569,152.921)(163.569,154.050)(164.569,155.176)(165.569,156.298)
(166.569,157.417)(167.569,158.532)(168.569,159.644)(169.569,160.753)
(170.569,161.860)(171.569,162.963)(172.569,164.064)(173.569,165.162)
(174.569,166.258)(175.569,167.351)(176.569,168.442)(177.569,169.531)
(178.569,170.618)(179.569,171.702)(180.569,172.785)(181.569,173.866)
(182.569,174.946)(183.569,176.023)(184.569,177.099)(185.569,178.173)
(186.569,179.246)(187.569,180.317)(188.569,181.386)(189.569,182.455)
(190.569,183.521)(191.569,184.587)(192.569,185.651)(193.569,186.714)
(194.569,187.776)(195.569,188.837)(196.569,189.896)(197.569,190.954)
(198.569,192.012)(199.569,193.068)(200.569,194.123)(201.569,195.178)
(202.569,196.231)(203.569,197.283)(204.569,198.335)(205.569,199.385)
(206.569,200.435)(207.569,201.484)(208.569,202.532)(209.569,203.580)
(210.569,204.626)(211.569,205.672)(212.569,206.717)(213.569,207.761)
(214.569,208.805)(215.569,209.848)(216.569,210.890)(217.569,211.932)
(218.569,212.973)(219.569,214.014)(220.569,215.053)(221.569,216.093)
(222.569,217.131)(223.569,218.169)(224.569,219.207)(225.569,220.244)
(226.569,221.280)(227.569,222.316)(228.569,223.352)(229.569,224.387)
(230.569,225.421)(231.569,226.455)(232.569,227.489)(233.569,228.522) 
(234.569,229.555)(235.569,230.587)(236.569,231.619)(237.569,232.650)
(238.569,233.681)(239.569,234.712)(240.569,235.742)(241.569,236.772)
(242.569,237.801)(243.569,238.830)(244.569,239.859)(245.569,240.887)
(246.569,241.915)(247.569,242.943)(248.569,243.970)(249.569,244.997)
(250.569,246.024)(251.569,247.050)(252.569,248.076)(253.569,249.102)
(254.569,250.127)(255.569,251.152)(256.569,252.177)(257.569,253.202)
(258.569,254.226)(259.569,255.250)(260.569,256.274)(261.569,257.297)
(262.569,258.320)(263.569,259.343)(264.569,260.366)(265.569,261.389)
(266.569,262.411)(267.569,263.433)(268.569,264.454)(269.569,265.476)
(270.569,266.497)(271.569,267.518)(272.569,268.539)(273.569,269.560)
(274.569,270.580)(275.569,271.600)(276.569,272.620)(277.569,273.640)
(278.569,274.660)(279.569,275.679)(280.569,276.698)(281.569,277.717)
(282.569,278.736)(283.569,279.755)(284.569,280.773)(285.569,281.791)
(286.569,282.809)(287.569,283.827)(288.569,284.845)(289.569,285.862)
(290.569,286.880)(291.569,287.897)(292.569,288.914)(293.569,289.931)
(294.569,290.948)(295.569,291.964)(296.569,292.981)(297.569,293.997)
(298.569,295.013)(299.569,296.029)(300.569,297.045)(301.569,298.061)
(302.569,299.076)(303.569,300.092)(304.569,301.107)(305.569,302.122)
(306.569,303.137)(307.569,304.152)(308.569,305.167)(309.569,306.181)
(310.569,307.196)(311.569,308.210)(312.569,309.224)(313.569,310.239)
(314.569,311.253)(315.569,312.266)(316.569,313.280)(317.569,314.294)
(318.569,315.307)(319.569,316.321)(320.569,317.334)(321.569,318.347)
(322.569,319.360)(323.569,320.373)(324.569,321.386)(325.569,322.399)
(326.569,323.412)(327.569,324.424)(328.569,325.437)(329.569,326.449)


\path(126.569,98.284)(127.069,94.765)(127.569,93.443)
(128.069,92.478)(128.569,91.697)(129.569,90.451) 
(130.569,89.461)(131.569,88.632)(132.569,87.916)(133.569,87.285)
(134.569,86.719)(135.569,86.207)(136.569,85.739)(137.569,85.308)
(138.569,84.909)(139.569,84.537)(140.569,84.190)(141.569,83.864)
(142.569,83.557)(143.569,83.267)(144.569,82.992)(145.569,82.731)
(146.569,82.483)(147.569,82.247)(148.569,82.022)(149.569,81.806)
(150.569,81.599)(151.569,81.401)(152.569,81.211)(153.569,81.028)
(154.569,80.852)(155.569,80.683)(156.569,80.520)(157.569,80.362)
(158.569,80.209)(159.569,80.062)(160.569,79.920)(161.569,79.782)
(162.569,79.648)(163.569,79.518)(164.569,79.392)(165.569,79.270)
(166.569,79.152)(167.569,79.036)(168.569,78.924)(169.569,78.815)
(170.569,78.709)(171.569,78.606)(172.569,78.505)(173.569,78.407)
(174.569,78.311)(175.569,78.218)(176.569,78.127)(177.569,78.038)
(178.569,77.951)(179.569,77.866)(180.569,77.783)(181.569,77.702)
(182.569,77.623)(183.569,77.546)(184.569,77.470)(185.569,77.396)
(186.569,77.323)(187.569,77.252)(188.569,77.182)(189.569,77.114)
(190.569,77.047)(191.569,76.982)(192.569,76.917)(193.569,76.854)
(194.569,76.793)(195.569,76.732)(196.569,76.672)(197.569,76.614)
(198.569,76.557)(199.569,76.500)(200.569,76.445)(201.569,76.391)
(202.569,76.338)(203.569,76.285)(204.569,76.234)(205.569,76.183)
(206.569,76.133)(207.569,76.084)(208.569,76.036)(209.569,75.989)
(210.569,75.942)(211.569,75.897)(212.569,75.852)(213.569,75.807)
(214.569,75.763)(215.569,75.720)(216.569,75.678)(217.569,75.637)
(218.569,75.595)(219.569,75.555)(220.569,75.515)(221.569,75.476)
(222.569,75.437)(223.569,75.399)(224.569,75.362)(225.569,75.325)
(226.569,75.288)(227.569,75.252)(228.569,75.217)(229.569,75.182)
(230.569,75.147)(231.569,75.113)(232.569,75.080)(233.569,75.047)
(234.569,75.014)(235.569,74.982)(236.569,74.950)(237.569,74.919)
(238.569,74.888)(239.569,74.857)(240.569,74.827)(241.569,74.797)
(242.569,74.768)(243.569,74.738)(244.569,74.710)(245.569,74.681)
(246.569,74.653)(247.569,74.626)(248.569,74.598)(249.569,74.572)
(250.569,74.545)(251.569,74.518)(252.569,74.492)(253.569,74.467)
(254.569,74.441)(255.569,74.416)(256.569,74.391)(257.569,74.367)
(258.569,74.342)(259.569,74.318)(260.569,74.295)(261.569,74.271)
(262.569,74.248)(263.569,74.225)(264.569,74.202)(265.569,74.180)
(266.569,74.158)(267.569,74.136)(268.569,74.114)(269.569,74.093)
(270.569,74.071)(271.569,74.050)(272.569,74.029)(273.569,74.009)
(274.569,73.988)(275.569,73.968)(276.569,73.948)(277.569,73.929)
(278.569,73.909)(279.569,73.890)(280.569,73.870)(281.569,73.851)
(282.569,73.833)(283.569,73.814)(284.569,73.796)(285.569,73.777)
(286.569,73.759)(287.569,73.741)(288.569,73.724)(289.569,73.706)
(290.569,73.689)(291.569,73.671)(292.569,73.654)(293.569,73.638)
(294.569,73.621)(295.569,73.604)(296.569,73.588)(297.569,73.571)
(298.569,73.555)(299.569,73.539)(300.569,73.524)(301.569,73.508)
(302.569,73.492)(303.569,73.477)(304.569,73.462)(305.569,73.446)
(306.569,73.431)(307.569,73.417)(308.569,73.402)(309.569,73.387)
(310.569,73.373)(311.569,73.358)(312.569,73.344)(313.569,73.330)
(314.569,73.316)(315.569,73.302)(316.569,73.288)(317.569,73.275)
(318.569,73.261)(319.569,73.248)(320.569,73.234)(321.569,73.221)
(322.569,73.208)(323.569,73.195)(324.569,73.182)(325.569,73.170)
(326.569,73.157)(327.569,73.144)(328.569,73.132)(329.569,73.120)
(330.569,73.107)(331.569,73.095)(332.569,73.083)(333.569,73.071)
(334.569,73.059)(335.569,73.047)(336.569,73.036)(337.569,73.024)
(338.569,73.013)(339.569,73.001)(340.569,72.990)(341.569,72.979)
(342.569,72.967)(343.569,72.956)(344.569,72.945)(345.569,72.934)
(346.569,72.923)(347.569,72.913)(348.569,72.902)(349.569,72.891)
(350.569,72.881)(351.569,72.870)(352.569,72.860)(353.569,72.850)
(354.569,72.840)(355.569,72.829)(356.569,72.819)(357.569,72.809)
(358.569,72.799)(359.569,72.790)


\path(269,16.368)
(270,30.000)(271,42.429)(272,53.818)(273,64.304)(274,74.000)
(275,83.000)(276,91.385)
\end{picture}
\hfil
\caption{An example of the worst case procedure. 
Spaces containing $u(x)$ and the non-linear term $R_0g(u)$ are
indicated by $\times$ and $\circ$, respectively; arrows stand
for embeddings, while dotted lines indicate new information on
$R_0g(u)$.} \label{wrst-fig} 
\end{figure}

For the case $s_0>s_1+\delta_1$ we explain our procedure in the following;
Figure~\ref{wrst-fig} illustrates the strategy. Observe first
that for $k=1$ the inequality
\begin{equation}
  s_0-\fracc n{p_0}\ge s_k+\delta_k-\fracc n{p_k}
  \label{ir63}
\end{equation}
may be either true or false. If it is false, the point
$(\fracc n{p_k},s_k+\delta_k)$ lies above the line of slope $1$ through
$(\fracc n{p_0},s_0)$, hence these points can be connected by a
straight line; this situation is treated further below in
Subsection~\ref{main-sssect} (and also
illustrated in Figure~\ref{wrst-fig}). We proceed
to show that \eqref{ir63} is false eventually for a certain choice of
the parameters $(s_j,p_j,q_j)$ for $j\ge2$.

Suppose therefore that for some $j\in\N$ we have shown that $u$ is in
a space $E^{s_j}_{p_j,q_j}$ fulfilling the inequality in \eqref{ir63}
and $\delta_j>0$. There
are three possibilities for the definition of
$(s_{j+1},p_{j+1},q_{j+1})$, cf.\ $1^\circ$--$3^\circ$ below that
apply in the given order (possibly $1^\circ$ or even
$1^\circ$ and $2^\circ$ is redundant). 

$1^\circ$ First we consider the case where 
\begin{align}
  \text{(I)} &\quad \fracc n{p_j}-\delta_j\ge0 \\
  \text{(II)}&\quad \fracc n{p_j}>\min(\fracc n{p_0},3+\sqrt8) 
\end{align}
both hold. Then we take a Sobolev embedding 
\begin{equation}
  E^{s_j+\delta_j}_{p_j,q_j}\hookrightarrow E^{s_j}_{p_{j+1},q_j}
  \label{ir64}
\end{equation}
with $\fracc n{p_{j+1}}=\fracc n{p_j}-\delta_j$; this is
possible since the inequalities $\infty\ge p_{j+1}>p_j$ follow from
(I) and $\delta_j>0$. Moreover we let 
\begin{equation}
  (s_{j+1},p_{j+1},q_{j+1})=(s_j,p_{j+1},q_j)
  \label{ir65}
\end{equation}
and it is seen that $s_j=s_1$ and $q_j=q_1$ result from \eqref{ir65}
for all $j$. By the definition of $(s_{j+1},p_{j+1},q_{j+1})$, and
since \eqref{ir63} for $k=j$ and $s_0>s_j$ are assumed to hold, it is
clear that we have 
$E^{s_0}_{p_0,q_0}\hookrightarrow E^{s_{j+1}}_{p_{j+1},q_{j+1}}$, and hence
\begin{equation}
  u= R_0f-R_0g(u)+\cal Ru\in E^{s_{j+1}}_{p_{j+1},q_{j+1}}.
  \label{ir66}    
\end{equation}
For this space containing $u$ we find
\begin{equation}
  s_{j+1}+\delta_{j+1}-\fracc n{p_{j+1}}=s_j+\delta_j-\fracc n{p_j}
   +\delta_{j+1}>s_j+\delta_j-\fracc n{p_j},
  \label{ir67}
\end{equation}
because by Theorem~\ref{dreg-thm} $\delta(s_1,\cdot)$ is a
non-decreasing function of $p$, so that the gain 
$\delta_{j+1}$ in \eqref{ir67} is bounded from below by the amount
$\delta_1>0$; in addition $\delta_j\in [\delta_1,2]$ since
$\sigma(s,p)\le s$. After finitely
many steps either (I) or (II) is false (because $\fracc n{p_j}$ is
decreasing with $j$), in which case we proceed by
$2^\circ$ and $3^\circ$, or \eqref{ir63} itself is false.

$2^\circ$ When (I) is false but (II) is true, 
$\fracc n{p_0}\le 3+\sqrt8$ (otherwise $3+\sqrt8< \fracc n{p_j}$ and
since $\delta_j$ is at most $2$, then (I) would be
true). Now a Sobolev embedding as above is impossible since (I) is
false, but we take a `shorter' one into 
$E^{s_j+\delta_j-\fracci n{p_j}}_{\infty,q_j}$ and let this 
have parameter $(s_{j+1},p_{j+1},q_{j+1})$. That $\fracc n{p_{j+1}}=0$
gives $\delta_{j+1}=2$, so 
\begin{equation}
  s_{j+1}+\delta_{j+1}-\fracc n{p_{j+1}}=s_j+\delta_j-\fracc n{p_j}+2
    \ge s_j+\delta_j-\fracc n{p_j}+\delta_1,
  \label{ir68}
\end{equation}
and the gain is at least $\delta_1$. This construction is at most
used once, for either it makes \eqref{ir63} false or it brings one to
the third case (since $\fracc n{p_{j+1}}=0$).

$3^\circ$ When (II) is false we observe first that $\delta(s,p_j)>0$ for
all $s>0$ if $\fracc n{p_j}<3+\sqrt{8}$.
Indeed, as noted after \eqref{C5}, $\max d(s)=(\sqrt{\fracc n{p_j}}-1)^2$ and 
\begin{equation}
  (\sqrt{\fracc n{p_j}}-1)^2=2\iff
  \sqrt{\fracc n{p_j}}=1+\sqrt{2}\iff
  \fracc n{p_j}=3+\sqrt{8},
\end{equation} 
so if $\fracc n{p_j}<3+\sqrt{8}$ we have
$d(s)<2$ for all $s\in\,]1;\fracc n{p_j}[$, and hence $\delta(s,p_j)\ge
2-\max d((\cdot))=:\alpha>0$ (regardless of whether $1<s<\fracc n{p_j}$ or
not).
 
Now if $\fracc n{p_j}=3+\sqrt{8}$ there is
the freedom to make a single Sobolev embedding of $E^{s_j}_{p_j,q_j}$
(thereby defining $(s_{j+1},p_{j+1},q_{j+1})$ without any gain), so we
can assume that 
\begin{equation}
  \fracc n{p_j}< 3+\sqrt{8},
\end{equation}
whenever (II) is
false. Then $\delta(s,p_j)>0$ for all $s>0$ as noted first.

Now we simply go upwards, that
means we let
\begin{equation}
  (s_{j+1},p_{j+1},q_{j+1})=(s_j+\delta_j,p_j,q_j).
  \label{ir69}
\end{equation}
Because \eqref{ir63} holds for $k=j$, there is an embedding
$E^{s_0}_{p_0,q_0}\hookrightarrow E^{s_j+\delta_j}_{p_j,q_j}$ since
also $p_j\ge p_0$ holds by the negation of (II). Again $u\in
E^{s_{j+1}}_{p_{j+1},q_{j+1}}$, only this time with a gain
$s_{j+1}+\delta_{j+1}-(s_j+\delta_j)=\delta_{j+1}$. Since $\fracc
n{p_k}=\fracc n{p_j}$ for all $k>j$ in this procedure, (II) remains
false; and we have $\delta_k\ge \alpha>0$ for
all $k$, so \eqref{ir63} is violated in a finite number of steps.

\bigskip

Consequently, when the $(s_j,p_j,q_j)$ are defined as above, then for a finite
$k$ the function $u(x)$ belongs to some $E^{s_k}_{p_k,q_k}$ for which
\eqref{ir63} false. 
Moreover, $(s_k,p_k,q_k)\in \Dm(A_T+g(\cdot))$, for it is clear (but
tedious to prove) that
this set is stable under $1^\circ$, $ 2^\circ$ and $3^\circ$ above.

However, this means that the considered case has been reduced to one
of those treated in the next subsection.

\subsubsection{The Main Argument}   \label{main-sssect}
We return to a sketch of the full proof, which eventually would go
through the same
cases as those considered in \cite{JJ95reg}; there a proper exposition for
problems of product-type is given. \cite{JJ95stjm} gives a concise
presentation of the ideas, which originated in \cite{JJ93}.

First of all, if $R_0f+\cal Ru\in E^{s_0}_{p_0,q_0}$ and $R_0g(u)\in
E^{s_1+\delta_1}_{p_1,q_1}$ in \eqref{ir53} with
\begin{equation}
  s_1+\delta_1\ge s_0 \quad\text{and}\quad s_1+\delta_1 -\fracc
   n{p_1}\ge s_0-\fracc n{p_0},
  \label{ir81}
\end{equation}
then there is actually an embedding
$E^{s_1+\delta_1}_{p_1,q_1}\hookrightarrow E^{s_0}_{p_0,q_0}$, so from
\eqref{ir53} it follows that $u\in E^{s_0}_{p_0,q_0}$ (as also used in
the beginning of Subsection~\ref{wrst-sssect}).

Secondly, there is the case with 
\begin{equation}
  s_1+\delta_1< s_0 \quad\text{and}\quad s_1+\delta_1 -\fracc
   n{p_1}\ge s_0-\fracc n{p_0}.
  \label{ir82}
\end{equation}
(This, and \eqref{ir81}, is the one that the worst case was reduced to in
Subsection~\ref{wrst-sssect} above.)  The spaces $E^{s_j}_{p_j,q_j}$
considered for this case in \cite{JJ95reg} all have $(\fracc
n{p_j},s_j)$ lying on or above each of the two lines $s=s_1+\delta_1$
and $s=\fracnp+s_0-\fracc n{p_0}$, so it is geometrically clear that
all these $(s_j,p_j,q_j)$ belong to $\Dm(A_T+g(\cdot))$. See also
Figure~\ref{wrst-fig} after the first horizontal arrow. Hence, by
\cite{JJ95reg}, we obtain $u\in E^{s_0}_{p_0,q_0}$. 

Thirdly, when
\begin{equation}
  s_1+\delta_1\ge s_0 \quad\text{and}\quad s_1+\delta_1 -\fracc
   n{p_1}< s_0-\fracc n{p_0},
  \label{ir83}
\end{equation}
already $(s_2,p_2,q_2)$ defined as in \eqref{ir63} may be
outside of $\Dm(A_T+g(\cdot))$ because the condition
$s_2>r+\frac1{p_2}-1$ may be violated. 

However, it is a main point of \cite{JJ93,JJ95stjm,JJ95reg} that such
problems can be overcome if $\delta(s,p)$ satisfies additional
conditions, and these can be verified in our case. (Phrased briefly,
$R_0g(\cdot)$ should be defined on $E^{s_2}_{p_2,q_2}$: when the problem
occurs for $(s_2,p_2,q_2)$, then $p_2>1$. For $r=1$, $R_0g(\cdot)$ makes
sense on $E^s_{p,q}$ as soon as $s>0$ and $p>1$, for $g(\cdot)$ has order
$0$ on $L_p$, where $R_0$ is defined; and if $r=2$, then $s_2>1$, and
$g(E^{s_2}_{p_2,q_2})\subset H^1_{p_2}$.) Non-convexity problems do not
occur either. 

Finally, when the spaces are such that 
\begin{equation}
  s_1+\delta_1< s_0 \quad\text{and}\quad s_1+\delta_1 -\fracc
   n{p_1}< s_0-\fracc n{p_0}
  \label{ir84}
\end{equation}
the procedure in \cite{JJ95reg} is just to go upwards as in
\eqref{ir69}. Evidently this may be inappropriate here if $\fracc
n{p_1}\ge 3+\sqrt8$, as one will hit the bulge defined by condition
(iii) in Theorem~\ref{dreg-thm}.

However, as described in the worst case analysis in \ref{wrst-sssect},
it is possible first to move left of $\fracnp=3+\sqrt8$ ($1^\circ$),
if necessary make sure that $\fracc n{p_j}<\fracc n{p_0}$ too
($2^\circ$), and then move upwards until a reduction to \eqref{ir81}
or \eqref{ir82}
is achieved ($3^\circ$, with an intermediate step if some $\fracc
n{p_j}$ equals $3+\sqrt8$).

In general the strategy of \cite{JJ95reg} in this case is
to move upwards until \eqref{ir84} is not valid any longer (with $s_j$
and $p_j$ replacing $s_1$ and $p_1$), thus obtaining a reduction to
the cases in \eqref{ir81},\eqref{ir82} and \eqref{ir83}.
The procedure in Subsection~\ref{wrst-sssect} serves the
same purpose, so the argument of \cite{JJ95reg} may be applied the
rest of the way to get $u\in E^{s_0}_{p_0,q_0}$ also in
this situation.

\bigskip

Finally, note that $\Dm(A_T+g(\cdot))$ is an open set defined
by sharp inequalities, so we can weaken the assumption on $u(x)$
slightly to begin with. Thus it is not a restriction to assume
$E^{s_1}_{p_1,q_1;T}=F^{s_1}_{p_1,\infty;T}$.

Since $f\in F^{s_0-2-\varepsilon}_{p_0,\infty}(\overline{\Omega})$ 
and $(s_0-\varepsilon,p_0,\infty)\in \Dm(A_T+g(\cdot))$ for $\varepsilon>0$
small enough, $u\in F^{s_0-\varepsilon}_{p_0,\infty}(\overline{\Omega})$
according to the proof given above. So by
\eqref{ir54} and the fact that $\sigma>s-2$ is possible near
$(s_0,p_0,q_0)$, we get $u\in E^{s_0}_{p_0,q_0;T}$.  

Altogether this completes the proof of Theorem~\ref{ireg-thm}. 

\begin{rem}  \label{parametrix-rem}
Although the basic formula \eqref{ir53} is not surprising, it has to be
derived in the indicated way, for if one rearranges \emph{before} the
application of $R_0$, then $R_0$ may be undefined   
on $E^{s_0-2}_{p_0,q_0}+E^{s_1-2}_{p_1,q_1}$ (that contains
$f-g(u)$). Moreover, in such cases the usual regularity statements for
elliptic problems cannot be used, so then it is necessary
to utilise the parametrix $R_0$.
\end{rem}

\section{The Existence Results} \label{ex-sect}
From the Leray--Schauder theorem we now deduce that
solutions exist as described in Theorem~\ref{solv-thm}.

It suffices to treat the case where the data space has the form 
\begin{equation}
  E^{s_1-2}_{p_1,q_1}\quad\text{for some $s_1<2$ and 
          $p_1$, $q_1\in\,]1,\infty]$}. 
  \label{ex51}
\end{equation}
To see this, we may for the actual data space $E^{s-2}_{p,q}$ use a
Sobolev embedding
\begin{equation}
  E^{s-2}_{p,q}\hookrightarrow E^{s_1-2}_{p_1,q_1},
  \quad\text{for $s-\fracnp=s_1-\fracc n{p_1}$, $q_1=q$}
  \label{ex52}
\end{equation}
when $s-\fracnp<2$ (since $s_1-2-\fracc n{p_1}<0$ in
\eqref{ex51}); for $s-\fracnp\ge2$ one can take
\begin{equation}
  E^{s-2}_{p,q}\hookrightarrow E^{-1/2}_{\infty,\infty}=: E^{s_1-2}_{p_1,q_1}.
  \label{ex53}
\end{equation}
For the corresponding solution spaces the inclusion
$E^s_{p,q;T}\subset L_{t-0}$ for $t^{-1}=(\fracp-\frac sn)_+$ carries
over to $E^{s_1}_{p_1,q_1;T}$ for the same $t$; that is, both (II) and
(III) are invariant under the reduction.

So when \eqref{ex51} is covered, there is to any $f\in E^{s-2}_{p,q}\subset
E^{s_1-2}_{p_1,q_1}$ a solution $u\in E^{s_1}_{p_1,q_1;T}$,
for it is easy to see that $(s_1,p_1,q_1)$ is or may be taken in
$\Dm(A_T+g(\cdot))$ (as for (i), $s_1$ should be taken in the gap between
the lines $s=r+\fracp-1$ and $s=r$ 
(then $p_1>1$ follows since $s-\fracnp>r-n$ by
(i)); (i) implies (ii), and (iii) is redundant for $s<3$). 

But then, from the assumption $(s,p,q)\in \Dm(A_T+g(\cdot))$, we infer from
Theorem~\ref{ireg-thm} that $u$ belongs to $E^{s}_{p,q;T}$.

\bigskip

So consider some $(s,p,q)$ in $\Dm(A_T+g(\cdot))$ with $s<2$ and
$1<p,q\le\infty$. 

When $A_T=A_T^*$ in $L_2$, the space $\ker A_T$
with $Q=P$
may be used as a range complement for $A_T$ for every $(s,p,q)$
according to Proposition~\ref{cmpl-prop}.
Moreover, with $Q^c=I-Q$ it is clear that $Q^c(E^s_{p,q;T})\subset
E^s_{p,q;T}$, and since $A_T$  by restriction is a bijection from
$Q^{c}(E^s_{p,q;T})$ to $Q^{c}(E^{s-2}_{p,q})$, there
is an inverse $B$ of this, that is
\begin{gather}
  B\colon Q^{c}(E^{s-2}_{p,q})\to Q^{c}(E^{s}_{p,q;T}),
  \\
  BA=1 \text{ on $Q^{c}(E^s_{p,q;T})$},\qquad
  AB=1 \text{ on $Q^{c}(E^{s-2}_{p,q})$}.
\end{gather}
These facts apply formally equally well to the case when $A_T$ is invertible.

Obviously $A_Tu+g(u)=f$ is equivalent to the system 
\begin{equation}
  \begin{aligned}
  v&=\lambda BQ^c(f-g(v+w))  \\
  w&=\lambda w+\lambda Q(f-g(v+w)) 
  \end{aligned}
  \label{bif-pb}
\end{equation}
when $\lambda=1$, $v=Q^cu$ and $w=Qu$. Here the transformation 
\begin{equation}
 (v,w)\mapsto
(BQ^c(f-g(v+w)), w+Q(f-g(v+w))) 
  \label{trf-eq}
\end{equation}
is continuous on $Q^c(E^s_{p,q;T})\times \ker A_T$ by Lemma~\ref{cont-lem} and
maps bounded sets to compact ones because 
$g(\cdot)$ does so from $E^s_{p,q}$ to $E^{s-2}_{p,q}$. So by the
Leray--Schauder theorem \eqref{bif-pb} is solvable for $\lambda=1$, if there
exist $c_1$ and $c_2$ in $\,]0,\infty[\,$ such that for 
every $\lambda\in [0,1]$ any solution satisfies
\begin{equation}
  \norm{v}{E^s_{p,q}}< c_1, 
   \qquad
  \norm{w}{E^s_{p,q}}< c_2.  
  \label{ex12}
\end{equation}
Assuming a solution of \eqref{bif-pb} does not exist for $\lambda=1$, then
$L_\infty\hookrightarrow E^{s-2}_{p,q}$ (which holds by \eqref{1.38} since
$s<2$) and \eqref{bif-pb} gives 
\begin{equation}
  \norm{v}{E^s_{p,q}}\le 
  c(\norm{f}{E^{s-2}_{p,q}}+\norm{g}{L_\infty})=: c_1;
  \label{ex13}
\end{equation}
hence $c_2$ does not exist. Thus there is for each $N\in\N$
a solution $(v_N,w_N)$ of \eqref{bif-pb} for some $\lambda_N\in ]0,1[$
such that 
\begin{equation}
  \norm{v_N}{E^s_{p,q}}< c_1 \quad\text{and}\quad
  \norm{w_N}{E^s_{p,q}}\ge N.
  \label{ex14}
\end{equation}
Passing to a subsequence if necessary, a sequence of solutions
$(v_k,t_k w_k)$ to \eqref{bif-pb} is found such that $ \norm{v_k}{E^s_{p,q}}<
c_1$ and 
\begin{equation}
 \norm{w_k}{L_\infty}=1,\qquad t_k\to\infty \quad\text{for }k\to\infty
 \label{ex15}
\end{equation}
Here it is used that all norms on $\ker A_T$ are equivalent. 
Furthermore, we can assume that for some $w_0\in \ker A_T$, 
\begin{equation}
  w_k\to w_0 \quad\text{in } L_\infty(\Omega);
  \label{ex17}
\end{equation}
indeed, by \eqref{ex15} a subsequence converges w$^*$ in $L_\infty$ and,
because $\ker{A_T}$ is finite dimensional, also uniformly
with limit $w_0$ in $\ker A_T$.

By \eqref{ex15}, $A_T$ is not invertible. Moreover, 
$\dual{Qf}{w_k}=\dual{f}{Qw_k}$
because $f$ may be approximated from $C^\infty(\overline{\Omega})$ and
because $Q$ is $L_2$-selfadjoint. With $W_k:=t_kw_k$, then the fact that
$(v_k,W_k)$ is a solution of \eqref{bif-pb} gives
\begin{equation}
  \int_\Omega W_k^2\,dx =\lambda_k\int_\Omega W_k^2\,dx
    -\lambda_k\int_\Omega Q(g(v_k+W_k)-f)W_k\,dx
  \label{ex21}
\end{equation}
or equivalently
\begin{equation}
  \int_\Omega g(v_k+W_k)w_k\,dx -\dual{f}{w_k}=
  \frac{\lambda_k-1}{\lambda_k t_k}\norm{W_k}{L_2}^2
  \label{ex22}
\end{equation}
Because $\lambda_k\in\,]0,1[$, the right hand side is strictly negative, so
since $(v_k)$ is bounded in $L_{t-0}$ and $k$ is arbitrary, 
(II) does not hold. 

Replacing $\lambda Q$ by $-\lambda Q$ in \eqref{bif-pb} yields
\eqref{ex22} with $1-\lambda_k$ instead of $\lambda_k-1$; hence (III)
does not hold either. The proof is complete.

\section{Final Remarks}

\begin{rem} \label{Dahlberg-rem}
  As mentioned in Section~\ref{thms-sect}, the function $\sigma(s,p)$
is conjectured to give the best possible smoothness index of
$E^\sigma_{p,q}$, the codomain of $g(\cdot)$ applied to
$E^{s}_{p,q}$, even for any $p$, $q\in\,]0,\infty]$ and any
$s>\max(0,\fracnp-n)$. 

On the one hand, for $1<s<\fracnp$, this is known to be correct if
e.g.\ $g(t)=\sin t$, for then when
$\varepsilon>0$ there exists $u_\varepsilon\in E^s_{p,q}$ with
$g(u_\varepsilon)\notin E^{\sigma(s,p)+\varepsilon}_{p,q}$. For this
we refer to \cite{Sic89} and the more extensive treatment in
\cite{RuSi96}.

On the other hand $g$ need not be periodic, cf.\ the classes
introduced in \cite{RuSi96}; there isn't complete freedom since
$g(t)=ct$  evidently acts on~$E^s_{p,q}$. 

However, for a subrange of
$1<s<\fracnp$, only this $g(t)$ has that property, as proved by Dahlberg
\cite{Dah79} for the $W^s_p$, and this function moreover falls outside
$\Cb^\infty(\R)$, in which we seek $g(t)$ in the present article.  
Thus it requires further knowledge on $g(t)$ to have another
boundary for the parameter domain $\Dm(A_T+g(\cdot))$ than the hyperbola
in (iii) of Theorem~\ref{dreg-thm}. 
\end{rem}

\begin{rem}[Quasi-Banach spaces]   \label{FLS-rem}
Our existence results are all based on the Leray--Schauder theorem,
although the spaces are merely quasi-Banach when $p<1$ or $ q<1$; but the
theorem was applied for $p$, $q>1$, for in \eqref{ex52} ff. we reduced to
this case by means of the regularity result in Theorem~\ref{ireg-thm}. 
However, the mapping degree has been extended to the
full Besov and Triebel--Lizorkin scales (although this was not used here),
cf.~\cite{FR87}.  

Theorems~\ref{ireg-thm} and \ref{solv-thm} are based on the linear 
elliptic theory in \cite{JJ96ell}, where the Fredholm properties for 
$p$ and $q\in\,]0,1[\,$ are obtained from a
reduction, this time by embeddings, to the Banach cases with $p$, $q>1$; cf.\
\cite[Rem.~5.1]{JJ96ell}. In addition one can extend the Fredholm concept
to quasi-Banach spaces with separating duals as in \cite{FR95}. 
\end{rem}

\begin{rem}[Continuity vs.~boundedness]  \label{cont-rem}
In the definition of $\Dm(A_T+g(\cdot))$ it suffices to require $g(\cdot)$
bounded $E^s_{p,q}\to E^{s-2}_{p,q}$, for this is the only relevant property
for whether $A_T$ or 
$g(\cdot)$ is the dominant operator. Hence continuity of $g(\cdot)$ is
not needed in Theorem~\ref{ireg-thm}, whereas it is for
Theorem~\ref{solv-thm}, in which case it is provided by Lemma~\ref{cont-lem}
at once.   
\end{rem}

\begin{rem}   \label{producttype-rem}
The present pseudo-differential approach to the inverse regularity
properties has predecessors for simpler problems of \emph{product-type},
primarily the stationary Navier--Stokes equations with various boundary
conditions, cf.\ \cite{JJ93,JJ95stjm,JJ95reg}. Comparisons with the
present problem are made in the beginning of Section~\ref{irpf-ssect}
and Subsections~\ref{wrst-sssect} and \ref{main-sssect}. 
\end{rem}

\begin{rem}[Data beyond the borderline]  \label{ireg-rem}
In Theorem~\ref{ireg-thm} the conclusion can be obtained
even for $f(x)$ in some $E^{s_0-2}_{p_0,q_0}$ 
\emph{outside} of $\Dm(A_T+g(\cdot))$, at least when $(s_1,p_1,q_1)\in
\Dm(A_T+g(\cdot))$ with $s_1>1$. More precisely, a range
of $(s_0,p_0,q_0)$ violating (iii) in
Theorem~\ref{dreg-thm} can then be treated. E.g.\ if 
$s_0<\sigma(s_1,p_1)+2$ this is trivial since
$E^{\sigma(s_1,p_1)+2}_{p_1,q_1}\hookrightarrow E^{s_0}_{p_0,q_0}$ in
\eqref{ir54} then.

More generally one could ask for $s_0>\sigma(s_1,p_1)+2$ with
$(s_0,p_0,q_0)$ outside of
$\Dm(A_T+g(\cdot))$. We have an argument based on interpolation and
composition estimates with fixed $s$ and variable $p$ that yields
$u\in E^{s_0}_{p_0,q_0}$ provided $(s_0,p_0,q_0)$ is close to
$\Dm(A_T+g(\cdot))$\,---\,but we omit the details here. 

However, this emphasises that \emph{direct} regularity properties
like those in Theorem~\ref{dreg-thm} and \emph{inverse} regularity
properties, of which there are some in Theorem~\ref{ireg-thm}, should
be analysed separately, since for non-linear problems these notions allow
different sets of parameters $(s,p,q)$ to be considered.
\end{rem}

\section*{Acknowledgements}
This work was done partly during the first author's stay at the
Friedrich--Schiller University of Jena, and J.~Johnsen is grateful for the
warm hospitality he enjoyed at the 
Mathematics Department there. In addition we thank 
W.~Sickel and S.~I.~Poho{\v z}aev for discussions on the subject.


\ifx\undefined\bysame
\newcommand{\bysame}{\leavevmode\hbox to3em{\hrulefill}\,}
\fi

\end{document}